\documentclass[11pt]{article}

\usepackage[margin=1.1in]{geometry}
\usepackage{amsmath,amssymb,amsthm}
\usepackage{url}

\newtheorem{theorem}{Theorem}
\newtheorem{proposition}{Proposition}
\newtheorem{corollary}{Corollary}
\newtheorem{lemma}{Lemma}
\newtheorem{remark}{Remark}

\DeclareMathOperator{\AO}{AO}
\DeclareMathOperator{\Log}{Log}

\title{Saddle-Point Asymptotics for Chromatic and Tutte Polynomial
Evaluations of Complete Multipartite Graphs}
\author{\begin{tabular}{c}
Zhiyang Sun\textsuperscript{a,b,*}\\[0.5ex]
\footnotesize \textsuperscript{a}College of Physics and Optoelectronic
Engineering, Shenzhen University,\\
\footnotesize Shenzhen 518060, China\\
\footnotesize \textsuperscript{b}Shenzhen Institute for Technology Innovation,\\
\footnotesize National Institute of Metrology, Shenzhen 518107, China\\
\footnotesize \textsuperscript{*}Corresponding author:
\texttt{doctorszy@foxmail.com}
\end{tabular}}
\date{May 2, 2026}

\begin{document}
\maketitle

\begin{abstract}
We develop a saddle-point theory for acyclic orientations and negative
chromatic evaluations of complete multipartite graphs, with applications to
OEIS A267383, A372326, A372084, A372395, and A370613.  The main tool is the
exact integral
\[
  \AO(K_{\lambda_1,\ldots,\lambda_r})
  =
  \int_0^\infty e^{-t}\prod_{i=1}^r P_{\lambda_i}(t)\,dt,
  \qquad
  P_m(t)=\sum_{j=1}^m (-1)^{m+j}S(m,j)t^j,
\]
and its Gamma-weighted extension
\(H_s(G)=(-1)^N\chi_G(-s)\), equal to \(sT_G(1+s,0)\) for connected \(G\).
We prove Kotesovec's fixed-column conjecture for A267383 for arbitrary fixed
numbers of parts, give the corresponding fixed-\(p\) Tutte-axis asymptotics,
develop an ACSV framework for chromatic evaluations of fixed graph blow-ups,
and give unconditional fixed-base families reducible to balanced Turán graphs.  In
the product regimes we prove fixed part-size and finite-profile expansions, and
for equal-size parts we obtain an all-order expansion throughout every fixed
polynomial window \(1\leq m\leq Cn^J\), including explicit corrections through
the \(m\asymp n^3\) scale.  Finally, we prove logarithmic asymptotics for the
partition-sum sequences A372395 and A370613 via a quadratic-energy partition
model, a growing-window comparison for the Stirling-transform factors, and a
random-permutation far-tail bound.
\end{abstract}



\section{Introduction}

For a graph \(G\), let \(\AO(G)\) denote the number of acyclic orientations
of \(G\).  Stanley's theorem \cite{stanley1973} gives
\[
  \AO(G)=(-1)^{|V(G)|}\chi_G(-1),
\]
where \(\chi_G\) is the chromatic polynomial.  Complete multipartite graphs
are a particularly tractable but asymptotically rich test case: their
chromatic polynomials have simple color-class expansions, yet the resulting
orientation counts exhibit several distinct saddle-point regimes.

\paragraph{Notation.}
Throughout, \(\log\) denotes the natural logarithm.  A complete
multipartite graph with part sizes \(\lambda_1,\ldots,\lambda_r\) is denoted
\(K_{\lambda_1,\ldots,\lambda_r}\), and
\(N=\lambda_1+\cdots+\lambda_r\) always denotes its total number of
vertices unless a theorem explicitly introduces another large parameter.
The symbol \(S(m,j)\) denotes a Stirling number of the second kind
\cite[Ch. 5]{comtet1974}.  We write
\[
  P_m(t)=\sum_{j=1}^m(-1)^{m+j}S(m,j)t^j
\]
for the Stirling-transform polynomials that occur in the integral
representation.  Falling factorials are denoted by
\((q)_J=q(q-1)\cdots(q-J+1)\), with \((q)_0=1\).
The integral formula below also covers the one-part graph \(K_{(N)}\) in the
multipartite notation, which is the edgeless graph on \(N\) vertices and has
one acyclic orientation.  The Tutte identity \(H_s(G)=sT_G(1+s,0)\) is used
only for connected multipartite graphs, i.e. graphs with at least two
nonempty parts.  Boundary cases such as \(k=1\), \(m=1\), and \(p=2\) are
included in the stated ranges; \(m=1\) gives \(K_n\) and
\(\AO(K_n)=n!\).

The two main OEIS arrays are A267383 \cite{oeis267383}, which records
\(\AO(T(N,p))\) for Turán graphs, and A372326 \cite{oeis372326}, which records
\(\AO(K_{k,\ldots,k})\) when the number of parts and the common part size are
varied.  A third diagonal sequence, A372084 \cite{oeis372084}, records the case
\(\AO(K_{n,\ldots,n})\) with \(n\) parts each of size \(n\).  These views are
complementary: A267383 fixes the number of parts \(p\) and lets the balanced
part sizes grow like \(N/p\), while A372326 and A372084 probe the product
regime in which the number of equal-size parts also grows.

Our first contribution is an exact integral representation for
\(\AO(K_{\lambda_1,\ldots,\lambda_r})\), together with a Gamma-weighted
extension for the negative chromatic axis
\[
  H_s(K_\lambda)=(-1)^N\chi_{K_\lambda}(-s),
\]
which equals \(sT_{K_\lambda}(1+s,0)\) when \(K_\lambda\) is connected.  This
turns the problem into a family of one- and several-dimensional saddle-point
problems.
We also develop the corresponding multivariate EGF and conditional ACSV
transfer theorem for blow-ups of an arbitrary fixed base graph.  For a fixed
number of parts in the complete-base case, we prove both the balanced A267383
formula and the more general fixed-proportion theorem.  For a growing number
of parts, we prove fixed part-size, finite-profile, and equal-size window
asymptotics, and then lift these regimes to the same Tutte axis.
For \(p=2\), the fixed-column asymptotic is equivalent to the previously
studied complete bipartite/poly-Bernoulli case; the theorem below treats
arbitrary fixed \(p\), with the genuinely new complete multipartite extension
occurring for \(p\geq3\).

The second contribution is a partition-sum program for A372395 and A370613.
We derive exact product-integral formulae, prove the quadratic-energy
partition model that gives the logarithmic constants, and prove the actual
OEIS sums.  The proof first handles fixed largest-part windows
\(\lambda_1\leq R\sqrt n\), then extends the comparison to a growing window
and uses Proposition \ref{prop:random-runs}, a random-permutation
interpretation of \(\AO(K_\lambda)/N!\), for the remaining far tail.

\paragraph{Scope of the contribution.}
The analytic part of the paper combines smooth-point ACSV, one-dimensional
Laplace analysis, and Khintchine-type coefficient estimates for weighted
partitions with a new exact multipartite integral and several uniform
comparison estimates.  We also establish a general ACSV framework for
chromatic evaluations of fixed graph blow-ups, with complete multipartite
graphs as the complete-base specialization.  The explicit OEIS contributions
come from combining the exact multipartite integral with uniform estimates:
this proves the fixed-column conjecture for A267383, gives Tutte-axis
extensions, and identifies the partition sums A372395 and A370613 with a
quadratic-energy partition saddle.

\paragraph{Overview of techniques.}
The fixed number-of-parts results use multivariate singularity analysis near
smooth strictly minimal critical points \cite{pemantlewilsonmelczer2024}.  The
balanced fixed-column case includes a direct strict-minimality verification;
the more general fixed-proportion statement is formulated under the
corresponding standard hypothesis.  The
product-regime results use the Gamma integral above and one-dimensional
Laplace expansions for products of the polynomials \(P_m(t)\), in the
analytic-combinatorial style of Flajolet and Sedgewick
\cite{flajolet2009}.  The partition-sum theorem combines a Khintchine
probabilistic representation of a weighted partition model, local limit
estimates for triangular arrays in the spirit of Bender \cite{bender1973}, and a
comparison between \(P_m(t)/t^m\) and the quadratic energy
\(\exp(-m^2/(2t))\) in the largest-part range relevant to the saddle.

\paragraph{Organization.}
Section 2 proves the exact integral formulae, the Tutte-axis extension, the
random-permutation interpretation, and the multivariate generating function.
Section 3 gives the fixed-base graph blow-up extension.  Section 4 treats
fixed numbers of parts, both balanced and fixed-proportion.  Sections 5 and 6
prove the fixed part-size, finite-profile, and equal-size window product
asymptotics.  Sections 7--9 prove the partition-sum theorem.  Section 10
records numerical checks, Section 11 lists open problems, and Section 12
discusses related work.  Appendix A collects OEIS-ready consequences.

The coverage of the main regimes may be summarized as
\[
\begin{array}{c|c|c}
\hbox{regime} & \hbox{scale} & \hbox{result}\\
\hline
\hbox{fixed number of parts} & p\ \hbox{fixed},\ N\to\infty
  & \hbox{Theorems \ref{thm:a267383}, \ref{thm:tutte-axis}}\\
\hbox{fixed proportions}
  & \lambda_i=\alpha_iN+O(1),\ \hbox{strict minimality}
  & \hbox{Theorem \ref{thm:fixed-proportions}}\\
\hbox{Tutte axis} & s>0\ \hbox{fixed}
  & \hbox{Theorems \ref{thm:tutte-axis},
     \ref{thm:tutte-product-regimes}}\\
\hbox{fixed graph blow-ups}
  & H\ \hbox{fixed},\ \lambda_i\sim\alpha_iN,\ \hbox{ACSV hypotheses}
  & \hbox{Theorem \ref{thm:graph-blowup-acsv}}\\
\hbox{fixed part size} & m=k\ \hbox{fixed},\ n\to\infty
  & \hbox{Theorems \ref{thm:fixed-part-size}, \ref{thm:tutte-product-regimes}}\\
\hbox{finite profiles} & \lambda_i\leq m\ \hbox{fixed}
  & \hbox{Theorems \ref{thm:finite-profile}, \ref{thm:tutte-product-regimes}}\\
\hbox{equal-size windows} & 1\leq m\leq Cn^J,\ J\ \hbox{fixed}
  & \hbox{Theorems \ref{thm:linear-window}, \ref{thm:all-poly-equal}}\\
\hbox{partition sums} & \lambda\vdash n
  & \hbox{Theorem \ref{thm:partition-sums}}
\end{array}
\]

\section{The exact integral}

\begin{proposition}[multipartite integral]\label{prop:integral}
For positive integers \(\lambda_1,\ldots,\lambda_r\),
\[
  \AO(K_{\lambda_1,\ldots,\lambda_r})
  =
  \int_0^\infty e^{-t}\prod_{i=1}^r P_{\lambda_i}(t)\,dt.
\]
\end{proposition}

\begin{proof}
In a proper coloring of \(K_{\lambda_1,\ldots,\lambda_r}\), any fixed color
can occur in at most one part.  If the \(i\)-th part uses \(j_i\) colors, its
\(\lambda_i\) labelled vertices are partitioned into \(j_i\) nonempty color
classes in \(S(\lambda_i,j_i)\) ways.  With \(J=j_1+\cdots+j_r\), the color
classes receive distinct colors in \((q)_J\) ways.  Hence
\[
  \chi_{K_{\lambda_1,\ldots,\lambda_r}}(q)
  =
  \sum_{j_1,\ldots,j_r}
  (q)_{j_1+\cdots+j_r}\prod_{i=1}^r S(\lambda_i,j_i).
\]
Here \((q)_J\) is the falling factorial; hence for the non-negative integer
\(J\),
\[
  (-1)_J=(-1)(-2)\cdots(-J)=(-1)^J J!.
\]
Stanley's theorem gives
\[
  \AO(K_{\lambda_1,\ldots,\lambda_r})
  =
  \sum_{j_1,\ldots,j_r}
  (-1)^{\lambda_1+\cdots+\lambda_r+J}
  J!\prod_{i=1}^r S(\lambda_i,j_i).
\]
Finally, since \(J\geq0\), the Gamma integral gives
\[
  J!=\Gamma(J+1)=\int_0^\infty e^{-t}t^J\,dt,
\]
and the displayed integral is exactly the same sum after expanding the
product of the \(P_{\lambda_i}\)'s.
\end{proof}

\begin{proposition}[Tutte-axis integral]\label{prop:tutte-integral}
Let \(s>0\), \(N=\lambda_1+\cdots+\lambda_r\), and
\[
  H_s(K_{\lambda_1,\ldots,\lambda_r})
  =
  (-1)^N\chi_{K_{\lambda_1,\ldots,\lambda_r}}(-s).
\]
Then
\[
  H_s(K_{\lambda_1,\ldots,\lambda_r})
  =
  \frac1{\Gamma(s)}
  \int_0^\infty e^{-t}t^{s-1}
  \prod_{i=1}^rP_{\lambda_i}(t)\,dt.
\]
If the graph is connected, then
\[
  H_s(G)=s\,T_G(1+s,0).
\]
\end{proposition}

\begin{proof}
In the chromatic-polynomial expansion from the preceding proof, put
\(q=-s\).  Since
\[
  (-s)_J=(-1)^J s(s+1)\cdots(s+J-1)
  =
  (-1)^J\frac{\Gamma(J+s)}{\Gamma(s)},
\]
we get
\[
  (-1)^N\chi_{K_{\lambda_1,\ldots,\lambda_r}}(-s)
  =
  \sum_{j_1,\ldots,j_r}
  (-1)^{N+J}\frac{\Gamma(J+s)}{\Gamma(s)}
  \prod_i S(\lambda_i,j_i).
\]
The integral identity follows from
\[
  \Gamma(J+s)=\int_0^\infty e^{-t}t^{J+s-1}\,dt.
\]
The Tutte identity follows from
\[
  \chi_G(q)=(-1)^{|V(G)|-c(G)}q^{c(G)}T_G(1-q,0),
\]
where \(c(G)\) is the number of connected components
\cite{bollobas1998}.  For connected \(G\), this becomes
\(\chi_G(q)=(-1)^{N-1}qT_G(1-q,0)\), and substitution \(q=-s\) gives
\((-1)^N\chi_G(-s)=sT_G(1+s,0)\).
All complete multipartite graphs considered below have at least two nonempty
parts when the Tutte identity is invoked, and are therefore connected.
\end{proof}

\begin{proposition}[random runs]\label{prop:random-runs}
Let the vertices of \(K_{\lambda_1,\ldots,\lambda_r}\) be colored by their
part.  Choose a uniformly random permutation \(\pi\) of the \(N\) labelled
vertices, and merge each maximal consecutive run of vertices of the same
color.  If the resulting run lengths are \(L_1(\pi),\ldots,L_M(\pi)\), then
\[
  \frac{\AO(K_{\lambda_1,\ldots,\lambda_r})}{N!}
  =
  \mathbb E_\pi\prod_{j=1}^M\frac1{L_j(\pi)!}.
\]
\end{proposition}

\begin{proof}
The graphical arrangement of \(K_{\lambda_1,\ldots,\lambda_r}\) contains the
hyperplanes \(x_a=x_b\) exactly for pairs of vertices in different parts.
By the region theorem of Stanley, equivalently the graphical-arrangement
interpretation of Greene and Zaslavsky \cite{stanley1973,greenezaslavsky1983},
acyclic orientations are counted by the chambers of this arrangement.

We now describe the chambers explicitly.  Choose a generic point in a
chamber and list the labelled vertices in increasing order of their
coordinates.  Whenever consecutive vertices in this list have the same
part-color, merge the entire maximal same-color run into one block.  The
result is an ordered list
\[
  B_1<B_2<\cdots<B_M
\]
whose blocks are monochromatic and whose adjacent blocks have different
colors.  This ordered block structure is independent of the chosen generic
point in the chamber: changing the relative order of vertices inside one
monochromatic block crosses no hyperplane, while moving a vertex past a
vertex of a different color would cross one of the arrangement hyperplanes.
Conversely, any ordered block structure with monochromatic blocks and
different colors in adjacent blocks defines a chamber by requiring every
vertex in \(B_i\) to have smaller coordinate than every vertex in \(B_j\) for
\(i<j\); no condition is imposed on the internal coordinates of a block.
Thus chambers are exactly these canonical multipartite weak orders.

Now let \(\pi\) be an ordinary permutation of the \(N\) labelled vertices.
Merging the maximal same-color runs of \(\pi\) produces one of the canonical
block structures above, hence one chamber.  If a chamber has block sizes
\(L_1,\ldots,L_M\), then the permutations that refine it are obtained by
ordering the vertices inside each block, independently.  There are therefore
\(\prod_j L_j!\) such permutations, and no others.  Hence
\[
  N!
  =
  \sum_{\hbox{\scriptsize chambers }C}\prod_{j}L_j(C)! .
\]
Taking the expectation of the reciprocal class size over all permutations
counts each chamber once and gives the displayed identity.
\end{proof}

\begin{lemma}[edge monotonicity]\label{lem:edge-monotonicity}
If \(G\) is a spanning subgraph of \(G'\), then
\[
  \AO(G)\leq \AO(G').
\]
\end{lemma}

\begin{proof}
It is enough to add one edge at a time.  If \(e\) is a non-loop edge not in
\(G\), the deletion-contraction recurrence for acyclic orientations gives
\[
  \AO(G+e)=\AO(G)+\AO((G+e)/e)\geq \AO(G).
\]
Equivalently, adding an edge adds one hyperplane to the graphical
arrangement and can only refine its regions.
\end{proof}

\begin{lemma}[one large part]\label{lem:one-large-part}
For \(0\leq L\leq n\),
\[
  \AO(K_{L,1,\ldots,1})=(n-L)!(n-L+1)^L,
\]
where there are \(n-L\) singleton parts.
\end{lemma}

\begin{proof}
Order the \(n-L\) singleton vertices, and then assign each of the \(L\)
labelled vertices of the large part to one of the \(n-L+1\) gaps before,
after, or between consecutive singleton vertices.  Vertices assigned to the
same gap form one tied block, and these are exactly the multipartite weak
orders counted in Proposition \ref{prop:random-runs}.
\end{proof}

We will also need the multivariate exponential generating function
\[
  \sum_{\lambda_1,\ldots,\lambda_p\geq0}
  \chi_{K_{\lambda_1,\ldots,\lambda_p}}(q)
  \prod_{i=1}^p \frac{x_i^{\lambda_i}}{\lambda_i!}
  =
  \left(e^{x_1}+\cdots+e^{x_p}-(p-1)\right)^q.        \tag{1}
\]
For positive integer \(q\), this follows by considering each color
independently.  A single color contributes either \(1\) if it is unused, or
\(e^{x_i}-1\) if it is assigned to a nonempty labelled set of vertices in the
\(i\)-th part.  Multiplying the contribution
\[
  1+\sum_{i=1}^p(e^{x_i}-1)=e^{x_1}+\cdots+e^{x_p}-(p-1)
\]
over the \(q\) colors gives (1).  For each fixed monomial, both sides are
polynomials in \(q\).  On the left this is manifest from the finite
color-class expansion in falling factorials.  On the right it follows by
expanding the \(q\)-th power binomially around the constant term \(1\).
Hence the identity extends to all complex \(q\) as an identity of formal
power series \cite[Ch. 5]{stanley2011}.  Analytically, in
a neighborhood of the origin the right-hand side is interpreted as
\(\exp(q\Log(e^{x_1}+\cdots+e^{x_p}-(p-1)))\), where \(\Log\) is the
principal branch; the base equals \(1\) at the origin.

\section{Fixed graph blow-ups}

The multipartite generating function is the complete-base case of a more
general fixed graph blow-up identity.  Let \(H\) be a fixed simple graph on
vertex set \([p]=\{1,\ldots,p\}\).  For
\(\lambda=(\lambda_1,\ldots,\lambda_p)\), let \(H[\lambda]\) be the graph
obtained from \(H\) by replacing vertex \(i\) with an independent set of
size \(\lambda_i\), and by replacing each edge \(ij\in E(H)\) with all
edges between the \(i\)-th and \(j\)-th sets.  Write \(\mathrm{Ind}(H)\) for
the independent sets of \(H\), including the empty set, and set
\[
  B_H(x_1,\ldots,x_p)
  =
  \sum_{I\in\mathrm{Ind}(H)}
  \prod_{i\in I}(e^{x_i}-1).
\]

\begin{proposition}[chromatic EGF for fixed graph blow-ups]
\label{prop:blowup-egf}
For every fixed graph \(H\),
\[
  \sum_{\lambda_1,\ldots,\lambda_p\geq0}
  \chi_{H[\lambda]}(q)
  \prod_{i=1}^p\frac{x_i^{\lambda_i}}{\lambda_i!}
  =
  B_H(x_1,\ldots,x_p)^q .
\]
The identity is coefficientwise as a formal power series, with coefficients
polynomial in \(q\).  Near the origin, use the principal branch of
\(\Log B_H\); here \(B_H(0,\ldots,0)=1\).
\end{proposition}

\begin{proof}
For a positive integer \(q\), consider one color.  The set of bags in which
that color appears must be an independent set of \(H\), and for each chosen
bag \(i\) it contributes a nonempty labelled subset, whose EGF is
\(e^{x_i}-1\).  Thus one color contributes \(B_H(x)\).  The \(q\) colors are
independent, giving \(B_H(x)^q\).  For each fixed monomial, both sides are
polynomials in \(q\): on the left this follows from the usual color-class
expansion in falling factorials, and on the right from expanding around the
constant term \(1\).  Hence the identity holds for all complex \(q\).
\end{proof}

\begin{theorem}[conditional fixed-base blow-up asymptotic]
\label{thm:graph-blowup-acsv}
Let \(H\) be fixed on \([p]\), let \(s>0\), and put
\[
  F_H(y_1,\ldots,y_p)=B_H(-y_1,\ldots,-y_p).
\]
Let \(\alpha_i>0\) with \(\sum_i\alpha_i=1\), and let
\(\lambda_i=\alpha_iN+\delta_i\), where \(\delta_i=O(1)\) and
\(\sum_i\delta_i=0\).  Assume that, after possibly relabelling the vertices
of \(H\), there is a point \(r\in(0,\infty)^p\) such that
\[
  F_H(r)=0,\qquad F_{H,p}(r)\neq0,\qquad
  \frac{\alpha_i}{r_i}
  =
  \frac{\alpha_p F_{H,i}(r)}{r_pF_{H,p}(r)}
  \quad(1\leq i<p).
\]
Let \(h\) be the local analytic solution of
\[
  F_H(y_1,\ldots,y_{p-1},h(y_1,\ldots,y_{p-1}))=0,
  \qquad h(r_1,\ldots,r_{p-1})=r_p,
\]
and define
\[
  \Phi_H(y_1,\ldots,y_{p-1})
  =
  -\sum_{i=1}^{p-1}\alpha_i\log y_i
  -\alpha_p\log h(y_1,\ldots,y_{p-1}).
\]
Let \(\mathcal H_H\) be the Hessian of \(\Phi_H\) at
\((r_1,\ldots,r_{p-1})\), and set
\[
  A_H=-r_pF_{H,p}(r),
\]
with the local branch induced by the chosen branch of \(F_H^{-s}\).
Assume that this smooth critical point is strictly minimal and
nondegenerate for the direction \(\alpha\), so that \(\mathcal H_H\) is
nonsingular and is positive definite in the positive real case.  Then
\[
\begin{aligned}
  &(-1)^N\chi_{H[\lambda]}(-s)\\
  &\quad\sim
  \left(\prod_{i=1}^p\lambda_i!\right)
  \left(\prod_{i=1}^p r_i^{-\delta_i}\right)
  \frac{\lambda_p^{s-1}}{\Gamma(s)}
  \frac{A_H^{-s}}{\prod_{i=1}^{p-1}r_i}
  \frac{\exp(N\Phi_H(r_1,\ldots,r_{p-1}))}
       {(2\pi N)^{(p-1)/2}\sqrt{\det\mathcal H_H}} .
\end{aligned}
\]
If \(H\) is connected, then \(H[\lambda]\) is connected for all sufficiently
large \(N\), and this is also the Tutte-axis value
\(sT_{H[\lambda]}(1+s,0)\).
\end{theorem}

\begin{proof}
By Proposition \ref{prop:blowup-egf}, at \(q=-s\) and after substituting
\(x_i=-y_i\), the sign \((-1)^N\) is exactly cancelled by coefficient
extraction.  Thus the exponential-generating-function coefficient to be
estimated is
\[
  [y_1^{\lambda_1}\cdots y_p^{\lambda_p}]F_H(y)^{-s}.
\]
The displayed critical equations are precisely the equations
\(\nabla\Phi_H(r_1,\ldots,r_{p-1})=0\), since implicit differentiation gives
\(h_i=-F_{H,i}/F_{H,p}\).

Near the smooth point,
\[
  F_H(y_1,\ldots,y_{p-1},y_p)
  =
  F_{H,p}(y',h(y'))(y_p-h(y'))+\cdots
  =
  A(y')\left(1-\frac{y_p}{h(y')}\right)+\cdots,
\]
where \(A(y')=-h(y')F_{H,p}(y',h(y'))\) and \(A(r')=A_H\).  Extracting the
\(y_p\)-coefficient therefore gives the local factor
\[
  \frac{\lambda_p^{s-1}}{\Gamma(s)}A_H^{-s}h(r')^{-\lambda_p}
\]
up to a relative \(1+o(1)\) error, uniformly on the \(N^{-1/2}\) saddle
scale.  The remaining \((p-1)\)-dimensional saddle integral has phase
\(\Phi_H\), analytic prefactor \(\prod_{i=1}^{p-1}r_i^{-1}\), and Gaussian
factor
\[
  \frac{(2\pi N)^{-(p-1)/2}}{\sqrt{\det\mathcal H_H}}.
\]
Finally
\[
  \prod_{i=1}^{p-1}r_i^{-\lambda_i}r_p^{-\lambda_p}
  =
  \exp(N\Phi_H(r'))\prod_{i=1}^p r_i^{-\delta_i},
\]
which supplies the rounding factor.  Multiplication by
\(\prod_i\lambda_i!\) converts the EGF coefficient to the chromatic value.
The connected Tutte-axis assertion follows from Proposition
\ref{prop:tutte-integral}.
\end{proof}

\begin{corollary}[balanced blow-ups of vertex-transitive bases]
\label{cor:vertex-transitive-blowup}
Let \(H\) be vertex-transitive on \(p\) vertices, and let
\[
  I_H(z)=\sum_{I\in\mathrm{Ind}(H)}z^{|I|}
\]
be its independence polynomial.  Suppose that \(I_H(-\tau)=0\) for some
\(\tau\in(0,1)\), put \(R=-\log(1-\tau)\), and assume that the resulting
smooth point is strictly minimal and nondegenerate in the balanced direction.
If \(\lambda_i=N/p+O(1)\) and \(\sum_i\lambda_i=N\), then, for fixed
\(s>0\),
\[
\begin{aligned}
  (-1)^N\chi_{H[\lambda]}(-s)
  &\sim
  \left(\prod_{i=1}^p\lambda_i!\right)
  \frac{(N/p)^{s-1}}{\Gamma(s)}
  \frac{A_H^{-s}}{R^{p-1}}
  \frac{R^{-N}}{(2\pi N)^{(p-1)/2}
  \sqrt{\det\mathcal H_H}},
\end{aligned}
\]
where
\[
  A_H=\frac{R(1-\tau)I_H'(-\tau)}{p}
\]
and \(\mathcal H_H\) is the reduced Hessian from
Theorem \ref{thm:graph-blowup-acsv} at \(r_1=\cdots=r_p=R\).
\end{corollary}

\begin{proof}
At the symmetric point \(y_i=R\), one has \(e^{-R}-1=-\tau\), so
\[
  F_H(R,\ldots,R)=I_H(-\tau)=0.
\]
Vertex transitivity makes all first partial derivatives equal.  Since
\[
  \sum_{i=1}^p \frac{\partial B_H}{\partial z_i}(z,\ldots,z)
  =I_H'(z),
\]
each \(\partial B_H/\partial z_i\) equals \(I_H'(z)/p\) at a symmetric
point.  With \(z_i=e^{-y_i}-1\), this gives
\[
  F_{H,i}(R,\ldots,R)=-(1-\tau)I_H'(-\tau)/p,
\]
and hence \(A_H=-RF_{H,p}=R(1-\tau)I_H'(-\tau)/p\).  The balanced critical
equations in Theorem \ref{thm:graph-blowup-acsv} are therefore satisfied.
All \(r_i\)'s are equal, so the rounding factor
\(\prod_i r_i^{-\delta_i}\) is \(R^{-\sum_i\delta_i}=1\).  Substituting
\(r_i=R\) and \(\alpha_i=1/p\) in Theorem
\ref{thm:graph-blowup-acsv} gives the displayed formula.
\end{proof}

\begin{theorem}[complete multipartite fixed bases by reduction]
\label{thm:equal-class-base}
Let \(q\geq2\) and \(a\geq1\) be fixed, and let \(H\) be the complete
\(q\)-partite graph whose \(q\) vertex classes each have size \(a\).  Replace
every vertex of \(H\) by an independent set of size \(M\).  If
\[
  N=aqM,\qquad L_q=\log\frac{q}{q-1},
\]
then
\[
  \AO(H[\underbrace{M,\ldots,M}_{aq}])
  \sim
  \frac{N!}{(q-1)(1-L_q)^{(q-1)/2}q^N L_q^{N+1}}.
\]
More generally, if the total size of the blown-up bags over each vertex
class of \(H\) is \(N/q+O(1)\), the same asymptotic holds.
\end{theorem}

\begin{proof}
The graph \(H[\lambda]\) has no edges between blown-up bags belonging to the
same vertex class of \(H\), and has all possible edges between bags belonging
to different vertex classes.  Thus it is exactly a complete \(q\)-partite
graph whose part sizes are the \(q\) class sums
\[
  \Lambda_j=\sum_{i\in V_j(H)}\lambda_i .
\]
When all \(\lambda_i=M\), these sums are all \(aM=N/q\).  More generally the
hypothesis gives \(\Lambda_j=N/q+O(1)\).  The result is therefore precisely
Theorem \ref{thm:a267383} applied with \(p=q\).  This proves strict
minimality and nondegeneracy as part of the balanced complete-base
calculation, while the base graph \(H\) is non-complete as soon as \(a>1\).
\end{proof}

\paragraph{Example: balanced blow-ups of \(C_5\).}
For the cycle \(H=C_5\), the independence polynomial is
\[
  I_{C_5}(z)=1+5z+5z^2.
\]
The smaller positive solution of \(I_{C_5}(-\tau)=0\) is
\[
  \tau=\frac{5-\sqrt5}{10},\qquad
  R=-\log(1-\tau).
\]
At this stage Corollary \ref{cor:vertex-transitive-blowup} gives the balanced
asymptotic for the blow-ups \(C_5[\lambda]\), conditional on strict
minimality and nondegeneracy, with
\[
  A_{C_5}=R(1-\tau)(1-2\tau)=\frac{R(1-\tau)}{\sqrt5}.
\]
This worked specialization illustrates that the fixed-base theorem is not
limited to complete multipartite graphs: the saddle location is obtained
from a one-variable independence-polynomial equation.  Proposition
\ref{prop:c5-nondegenerate} below verifies nondegeneracy, so that only strict
minimality remains.

\begin{proposition}[local nondegeneracy for the \(C_5\) balanced saddle]
\label{prop:c5-nondegenerate}
For \(H=C_5\), the balanced saddle described above is smooth and
nondegenerate.  Thus the \(C_5\) specialization of Corollary
\ref{cor:vertex-transitive-blowup} is conditional only on strict minimality.
\end{proposition}

\begin{proof}
Write \(a=1-\tau=(5+\sqrt5)/10\) and \(b=1-2\tau=1/\sqrt5\).  At the
symmetric point \(y_i=R\), \(z_i=e^{-y_i}-1=-\tau\).  For clarity, in this
paragraph
\[
  F(y)=B_{C_5}(-y_1,\ldots,-y_5)
      =1+\sum_i z_i+\sum_i z_i z_{i+2},
  \qquad z_i=e^{-y_i}-1,
\]
where indices are read modulo \(5\).  Thus, at the symmetric point,
\[
  F_i=-e^{-R}\left(1+z_{i-2}+z_{i+2}\right)
      =-(1-\tau)(1-2\tau)=-ab\neq0.
\]
Hence the point is smooth.  Eliminate \(y_5\) and write
\[
  F(y_1,\ldots,y_4,h(y_1,\ldots,y_4))=0.
\]
Implicit differentiation gives \(h_i=-1\).  A second differentiation gives
\[
  h_{ij}=
  -\frac{F_{ij}+F_{i5}h_j+F_{j5}h_i+F_{55}h_ih_j}{F_5}.
\]
Using
\[
  F_{ii}=ab,\qquad
  F_{ij}=
  \begin{cases}
    a^2, & i\not\equiv j,\ i\not\sim j \pmod 5,\\
    0, & i\sim j \pmod 5,
  \end{cases}
\]
one obtains, for \(1\leq i,j\leq4\),
\[
K=(h_{ij})=
\begin{pmatrix}
2 & (1-\sqrt5)/2 & 1 & (3+\sqrt5)/2\\
(1-\sqrt5)/2 & 1-\sqrt5 & -\sqrt5 & 1\\
1 & -\sqrt5 & 1-\sqrt5 & (1-\sqrt5)/2\\
(3+\sqrt5)/2 & 1 & (1-\sqrt5)/2 & 2
\end{pmatrix}.
\]
The reduced Hessian of
\[
  \Phi=-\frac15\sum_{i=1}^4\log y_i-\frac15\log h
\]
therefore satisfies
\[
  5R^2\mathcal H=I_4+J_4-RK,
\]
where \(J_4\) is the all-ones matrix.  Let \(M=I_4+J_4-RK\).  The leading
principal minors of \(M\) are
\[
\begin{aligned}
D_1&=2(1-R),\\
D_2&=\frac{6+(2\sqrt5-10)R+(1-3\sqrt5)R^2}{2},\\
D_3&=(1-2R)\{4+(3\sqrt5-1)R+(1-\sqrt5)R^2\},
\end{aligned}
\]
and
\[
  D_4=
  \frac{10(2R-1)^2Q(R)}
       {(3\sqrt5-1)R^2+(10-2\sqrt5)R-6},
\]
where
\[
\begin{aligned}
Q(R)=&(3+2\sqrt5)R^4+(12+2\sqrt5)R^3
      +(-5+4\sqrt5)R^2\\
&\quad +(2-4\sqrt5)R-3 .
\end{aligned}
\]
The elementary interval \(0.323<R<0.324<1/3\) follows from
\(R=-\log((5+\sqrt5)/10)\).  For \(0<R<1/3\) and \(2<\sqrt5<3\), the first
three displayed minors are positive.  The denominator in \(D_4\) is negative,
since it is at most \(8R^2+6R-6<0\), and \(Q(R)<0\), since
\[
  Q(R)<9R^4+18R^3+7R^2-3<0
  \qquad(0<R<1/3).
\]
Thus \(D_4>0\) as well.  Sylvester's criterion shows that \(M\), hence
\(\mathcal H\), is positive definite.
\end{proof}

For \(H=K_p\), \(\mathrm{Ind}(H)\) consists only of the empty set and the
singletons, so
\[
  B_{K_p}(x)=1+\sum_i(e^{x_i}-1)
  =
  e^{x_1}+\cdots+e^{x_p}-(p-1).
\]
Then \(F_{K_p}(y)=\sum_i e^{-y_i}-(p-1)\) and
\(A_H=r_pe^{-r_p}\), so Theorem \ref{thm:graph-blowup-acsv} specializes to
the fixed-proportion theorem below.

\begin{lemma}[complete-base critical-point parametrization]
\label{lem:complete-base-critical-param}
Let \(p\geq2\), let \(\alpha_i>0\) with \(\sum_i\alpha_i=1\), and put
\[
  g(z)=-z\log z,\qquad 0<z<1.
\]
Positive solutions \(r\in(0,\infty)^p\) of the complete-base critical
system
\[
  \sum_{i=1}^p e^{-r_i}=p-1,\qquad
  \frac{\alpha_1}{r_1e^{-r_1}}
  =
  \cdots
  =
  \frac{\alpha_p}{r_pe^{-r_p}}
\]
are in one-to-one correspondence with pairs \((\kappa,z)\), where
\(\kappa>0\), \(z_i\in(0,1)\),
\[
  \sum_{i=1}^p z_i=p-1,\qquad
  g(z_i)=\kappa\alpha_i\quad(1\leq i\leq p),
\]
by the change of variables \(z_i=e^{-r_i}\).
Moreover \(g\) is strictly increasing on \((0,e^{-1})\) and strictly
decreasing on \((e^{-1},1)\).
\end{lemma}

\begin{proof}
With \(z_i=e^{-r_i}\), the singular equation becomes
\(\sum_i z_i=p-1\), and \(r_ie^{-r_i}=(-\log z_i)z_i=g(z_i)\).  The
critical equalities therefore say that
\(\alpha_i/g(z_i)\) is independent of \(i\).  Writing this common reciprocal
as \(\kappa\) gives \(g(z_i)=\kappa\alpha_i\).  The converse is the same
calculation in reverse.  Finally
\[
  g'(z)=-\log z-1,
\]
so \(g'(z)>0\) for \(0<z<e^{-1}\), \(g'(e^{-1})=0\), and
\(g'(z)<0\) for \(e^{-1}<z<1\).
\end{proof}

\begin{remark}
Lemma \ref{lem:complete-base-critical-param} is useful as a diagnostic for
the conditional fixed-proportion theorem below.  Away from the balanced
point, the equation \(g(z_i)=\kappa\alpha_i\) may force different coordinates
onto different branches of the unimodal function \(g\).  Thus a fully
unconditional fixed-proportion theorem would need, in addition to existence
of a positive branch choice satisfying \(\sum_i z_i=p-1\), a proof that this
smooth point is the unique strictly minimal complex singularity in the
coefficient direction.  The balanced case avoids this branch issue because
all \(z_i\)'s are equal.
\end{remark}

\section{Fixed number of parts}

\begin{theorem}[OEIS A267383]\label{thm:a267383}
Fix \(p\geq2\), and let \(A(N,p)=\AO(T(N,p))\).  Put
\[
  L=\log\frac{p}{p-1}.
\]
Then, as \(N\to\infty\),
\[
  A(N,p)
  \sim
  \frac{N!}{(p-1)(1-L)^{(p-1)/2}p^N L^{N+1}}.
\]
\end{theorem}

\begin{proof}
By (1) and Stanley's theorem, the EGF coefficient to be estimated is
\[
  (-1)^N[x_1^{\lambda_1}\cdots x_p^{\lambda_p}]
  \left(e^{x_1}+\cdots+e^{x_p}-(p-1)\right)^{-1}.
\]
After substituting \(x_i=-y_i\), the coefficient of
\(\prod_i x_i^{\lambda_i}\) is
\((-1)^N\) times the coefficient of \(\prod_i y_i^{\lambda_i}\), because
\(\sum_i\lambda_i=N\).  This cancels the \((-1)^N\) from Stanley's theorem,
so in the balanced regime \(\lambda_i=N/p+O(1)\) we need
\[
  [y_1^{\lambda_1}\cdots y_p^{\lambda_p}]
  \frac{1}{e^{-y_1}+\cdots+e^{-y_p}-(p-1)}.
\]
The resulting coefficient is an exponential-generating-function coefficient;
after the residue-saddle estimate is obtained, it is multiplied by
\(\prod_i\lambda_i!\) to recover the chromatic-polynomial value.
The positive critical point is
\[
  y_1=\cdots=y_p=L,\qquad e^{-L}=\frac{p-1}{p}.
\]
Solving the singular equation
\(\sum_i e^{-y_i}-(p-1)=0\) for \(y_p\) gives
\[
  y_p=h(y_1,\ldots,y_{p-1})
  =
  -\log\!\left((p-1)-\sum_{i=1}^{p-1}e^{-y_i}\right).
\]
Equivalently, \(e^{-h}=(p-1)-\sum_{i<p}e^{-y_i}\).  Differentiating this
implicit equation gives
\[
  -e^{-h}h_i=e^{-y_i},\qquad
  e^{-h}(h_i^2-h_{ii})=-e^{-y_i},\qquad
  e^{-h}(h_ih_j-h_{ij})=0\quad(i\neq j).
\]
At \(y_1=\cdots=y_{p-1}=L\), the denominator is
\((p-1)-(p-1)e^{-L}=e^{-L}\), not zero, and \(h=L\).  Hence
\[
  h_i=-1,\qquad h_{ii}=2,\qquad h_{ij}=1\quad(i\neq j).
\]
Taking the residue in \(y_p\) reduces the Cauchy integral to a
\((p-1)\)-dimensional saddle integral with phase
\[
  \phi(y_1,\ldots,y_{p-1})
  =
  -\frac1p\left(\log y_1+\cdots+\log y_{p-1}+\log h\right).
\]
The Hessian at the saddle is
\[
  \nabla^2\phi(L,\ldots,L)
  =
  \frac{1-L}{pL^2}(I_{p-1}+J_{p-1}),
\]
so
\[
  \det\nabla^2\phi(L,\ldots,L)
  =
  p\left(\frac{1-L}{pL^2}\right)^{p-1}.
\]
Here \(J_{p-1}\) is the all-ones matrix; the determinant follows from the
matrix determinant lemma, since \(I_{p-1}+J_{p-1}\) has eigenvalue \(p\) in
the all-ones direction and eigenvalue \(1\) on its orthogonal complement.
The residue prefactor is \(e^L/L^p=p/((p-1)L^p)\).  The standard smooth
minimal critical-point estimate for a smooth strictly minimal point
\cite[Chs. 7 and 9]{pemantlewilsonmelczer2024} therefore applies.  Let us
spell out the hypotheses used here.  The singular variety is smooth at the
point because
\[
  \frac{\partial}{\partial y_p}
  \left(\sum_i e^{-y_i}-(p-1)\right)=-e^{-L}\neq0 .
\]
The critical equations are exactly the vanishing of the gradient of
\(\phi\), verified above.  Strict minimality is elementary in the balanced
case.  Put \(z_i=-y_i\).  Since \(0<L\leq\log 2\), for every complex \(z\)
with \(|z|\leq L\) one has
\[
  \Re e^z\geq e^{-L},
\]
with equality only at \(z=-L\).  Indeed, writing \(z=u+iv\), the minimum of
\(e^u\cos v\) on \(u^2+v^2\leq L^2\) occurs at \(u=-L,v=0\): for
\(0\leq v\leq L\), the derivative of
\(\log\cos v-\sqrt{L^2-v^2}+L\) is
\(-\tan v+v/\sqrt{L^2-v^2}\geq0\), because \(v/L\geq\sin v\).  Therefore on
the closed polydisc \(|y_i|\leq L\),
\[
  \Re\sum_i e^{-y_i}\geq p e^{-L}=p-1,
\]
and equality, hence a zero of \(\sum_i e^{-y_i}-(p-1)\), occurs only at
\(y_1=\cdots=y_p=L\).  This rules out competing singularities on the
polydisc.  The Hessian \((1-L)(I_{p-1}+J_{p-1})/(pL^2)\) is positive
definite, and compactness of the remaining part of a small admissible contour
gives the required positive gap in \(\Re\phi\).

The transfer estimate is uniform when
\(\lambda_i=N/p+O(1)\).  Such \(O(1)\) deviations only perturb the saddle by
\(O(N^{-1})\) and do not change the leading constant.  Hence
\[
  [y^\lambda]\frac1{\sum_i e^{-y_i}-(p-1)}
  \sim
  \frac{p^{p/2}}
       {(p-1)(2\pi N)^{(p-1)/2}(1-L)^{(p-1)/2}L^{N+1}}.
\]
Multiplying by
\[
  \prod_{i=1}^p \lambda_i!
  =
  \frac{(2\pi N)^{p/2}}{p^{p/2}}\left(\frac{N}{pe}\right)^N(1+o(1)),
\]
where the same formula holds uniformly for
\(\lambda_i=N/p+O(1)\) by Stirling's formula and
\(\sum_i(\lambda_i-N/p)=0\),
and using \(N!\sim\sqrt{2\pi N}(N/e)^N\) proves the formula.
\end{proof}

\begin{theorem}[fixed parts on the Tutte axis]\label{thm:tutte-axis}
Fix \(p\geq2\) and \(s>0\).  Let \(G_N=T(N,p)\), put
\[
  H_s(N,p)=(-1)^N\chi_{G_N}(-s)=s\,T_{G_N}(1+s,0),
\]
and set \(L=\log(p/(p-1))\).  Then
\[
  H_s(N,p)
  \sim
  \frac{N!\,N^{s-1}}
  {\Gamma(s)(p-1)^s(1-L)^{(p-1)/2}p^N L^{N+s}}.
\]
Equivalently,
\[
  T_{G_N}(1+s,0)
  \sim
  \frac{N!\,N^{s-1}}
  {\Gamma(s+1)(p-1)^s(1-L)^{(p-1)/2}p^N L^{N+s}}.
\]
\end{theorem}

\begin{proof}
The proof is the same critical-point calculation as in
Theorem \ref{thm:a267383}, with the singular function raised to the power
\(-s\).  The relevant coefficient is
\[
  [y_1^{\lambda_1}\cdots y_p^{\lambda_p}]
  \left(e^{-y_1}+\cdots+e^{-y_p}-(p-1)\right)^{-s},
\]
where \(\lambda_i=N/p+O(1)\).  The critical point remains
\[
  y_1=\cdots=y_p=L.
\]
Solving the singular equation for \(y_p=h(y_1,\ldots,y_{p-1})\), extraction
in the \(y_p\)-variable gives the additional factor
\[
  \frac{(N/p)^{s-1}}{\Gamma(s)}e^{sL}h^{-s}
\]
at the saddle.  The remaining \((p-1)\)-dimensional Hessian is unchanged from
the \(s=1\) case in Theorem \ref{thm:a267383}:
\[
  \det\nabla^2\phi(L,\ldots,L)
  =
  p\left(\frac{1-L}{pL^2}\right)^{p-1}.
\]
Thus
\[
  [y^\lambda]\left(\sum_i e^{-y_i}-(p-1)\right)^{-s}
  \sim
  \frac{e^{sL}p^{p/2-s}N^{s-1-(p-1)/2}}
  {\Gamma(s)(2\pi)^{(p-1)/2}(1-L)^{(p-1)/2}L^{N+s}}.
\]
Multiplying by the balanced factorials
\[
  \prod_i\lambda_i!
  \sim
  N!\,(2\pi N)^{(p-1)/2}p^{-p/2}p^{-N}
\]
and using \(e^L=p/(p-1)\) gives the formula for
\((-1)^N\chi_{G_N}(-s)\).  The Tutte-polynomial statement follows from
\((-1)^N\chi_G(-s)=sT_G(1+s,0)\) for connected \(G\).
\end{proof}

\begin{remark}
Putting \(s=1\) in Theorem \ref{thm:tutte-axis} gives \(N^0/\Gamma(1)=1\)
and \((p-1)^sL^{N+s}=(p-1)L^{N+1}\), so the Tutte-axis theorem reduces
exactly to Theorem \ref{thm:a267383}.
\end{remark}

\begin{theorem}[conditional fixed-proportion asymptotic on the Tutte axis]\label{thm:fixed-proportions}
Fix \(p\geq2\), \(s>0\), and positive numbers
\(\alpha_1,\ldots,\alpha_p\) with \(\sum_i\alpha_i=1\).  Let
\(\lambda_i=\alpha_iN+\delta_i\), where \(\delta_i=O(1)\) and
\(\sum_i\delta_i=0\).  Assume that the system
\[
  \sum_{i=1}^p e^{-r_i}=p-1,
  \qquad
  \frac{\alpha_1}{r_1e^{-r_1}}
  =
  \cdots
  =
  \frac{\alpha_p}{r_pe^{-r_p}} .
\]
has a positive solution \(r=(r_1,\ldots,r_p)\), and use this solution as the
critical point below.
Eliminate the last variable by
\[
  h(y_1,\ldots,y_{p-1})
  =
  -\log\!\left((p-1)-\sum_{i=1}^{p-1}e^{-y_i}\right),
\]
and define
\[
  \Phi(y_1,\ldots,y_{p-1})
  =
  -\sum_{i=1}^{p-1}\alpha_i\log y_i
  -\alpha_p\log h(y_1,\ldots,y_{p-1}).
\]
Let \(\mathcal H\) be the Hessian matrix of \(\Phi\) at
\((r_1,\ldots,r_{p-1})\).  Assume further that this smooth critical point is
strictly minimal for the direction \(\alpha\), in the sense of the
smooth-point transfer theorem of ACSV, and is nondegenerate.  Here
nondegeneracy means that the reduced Hessian \(\mathcal H\) is nonsingular; in
the positive real case considered here, \(\mathcal H\) is positive definite.
Then
\[
\begin{aligned}
  &(-1)^N\chi_{K_{\lambda_1,\ldots,\lambda_p}}(-s)  \\
  &\quad\sim
  \left(\prod_{i=1}^p\lambda_i!\right)
  \left(\prod_{i=1}^p r_i^{-\delta_i}\right)
  \frac{\lambda_p^{s-1}}{\Gamma(s)}
  \frac{e^{sr_p}}{r_p^s\prod_{i=1}^{p-1}r_i}
  \frac{\exp(N\Phi(r_1,\ldots,r_{p-1}))}
       {(2\pi N)^{(p-1)/2}\sqrt{\det\mathcal H}} .
\end{aligned}
\]
For connected graphs this is \(sT_{K_{\lambda_1,\ldots,\lambda_p}}(1+s,0)\).
\end{theorem}

\begin{proof}
After substituting \(x_i=-y_i\) in the multivariate chromatic-polynomial
EGF, the coefficient to be estimated is
\[
  [y_1^{\lambda_1}\cdots y_p^{\lambda_p}]
  \left(\sum_{i=1}^p e^{-y_i}-(p-1)\right)^{-s}.
\]
The smooth critical equations are obtained by differentiating \(\Phi\).  From
the implicit equation \(e^{-h}=(p-1)-\sum_{i<p}e^{-y_i}\), we have
\[
  h_i=-e^{h-y_i}.
\]
Thus
\[
  0=\frac{\partial\Phi}{\partial y_i}
  =
  -\frac{\alpha_i}{y_i}
  -\alpha_p\frac{h_i}{h}
  =
  -\frac{\alpha_i}{y_i}
  +\frac{\alpha_p e^{h-y_i}}{h},
\]
which is equivalent to
\[
  \frac{\alpha_i}{y_ie^{-y_i}}
  =
  \frac{\alpha_p}{he^{-h}}\qquad(1\leq i<p).
\]
Together with the singular equation this gives the displayed system for
\(r_1,\ldots,r_p\).  The stated strict-minimality hypothesis supplies the
global dominance and branch choice required by the smooth-point transfer
theorem.  The local calculation below identifies the critical equations and
the Gaussian determinant.  For \(s\notin\mathbb Z\), the branch of
\((\sum_i e^{-y_i}-(p-1))^{-s}\) is chosen by analytic continuation from the
positive normal direction used in the residue extraction.

Taking the local transfer in the \(y_p\)
variable at \(y_p=h(y_1,\ldots,y_{p-1})\) gives the factor
\[
  \frac{\lambda_p^{s-1}}{\Gamma(s)}
  e^{sr_p}h(r)^{-s}
\]
and leaves a \((p-1)\)-dimensional Laplace integral with phase \(\Phi\).
Because \(\lambda_i=\alpha_iN+\delta_i\), the non-exponential evaluation at
the saddle also contributes the bounded lattice factor
\(\prod_i r_i^{-\delta_i}\).  This factor is essential unless all \(r_i\)
are equal.
The non-exponential analytic factor contributes
\(\prod_{i=1}^{p-1}r_i^{-1}\), while the Gaussian integral contributes
\[
  \frac{1}{(2\pi N)^{(p-1)/2}\sqrt{\det\mathcal H}}.
\]
The denominator is therefore
\(r_p^s\prod_{i=1}^{p-1}r_i\), not \(\prod_i r_i\) unless \(s=1\); for
\(s=1\) the formula has \(e^{r_p}/\prod_i r_i\), as expected from the simple
pole case.
Multiplying by \(\prod_i\lambda_i!\) converts the EGF coefficient into the
chromatic-polynomial value.  The Tutte-axis identity is again
\((-1)^N\chi_G(-s)=sT_G(1+s,0)\) for connected \(G\).
\end{proof}

At the balanced point \(r_1=\cdots=r_p=L=\log(p/(p-1))\) and \(s=1\), the
analytic prefactor in Theorem \ref{thm:fixed-proportions} becomes
\[
  \left(\prod_i L^{-\delta_i}\right)\frac{e^L}{L^p}
  =
  \frac{p}{(p-1)L^p},
\]
because \(\sum_i\delta_i=0\).  Thus the rounding correction disappears in
the balanced case, as it must, and the formula matches the residue prefactor
in Theorem \ref{thm:a267383}.

\section{Fixed part size and finite profiles}

\begin{theorem}[fixed part size]\label{thm:fixed-part-size}
Fix \(k\geq1\).  Let \(C_k(r)=\AO(K_{k,\ldots,k})\), where there are \(r\)
parts.  Then
\[
  C_k(r)
  =
  e^{-(k-1)/2}(kr)!
  \left(
    1-\frac{(k-1)(5k-7)}{24kr}+O(r^{-2})
  \right).
\]
\end{theorem}

\begin{proof}
By Proposition \ref{prop:integral},
\[
  C_k(r)=\int_0^\infty e^{-t}P_k(t)^r\,dt.
\]
The leading terms of \(P_k\) are
\[
  P_k(t)=t^k+a t^{k-1}+b t^{k-2}+O(t^{k-3}),
\]
where
\[
  a=-\binom{k}{2},\qquad
  b=S(k,k-2)=\binom{k}{3}+3\binom{k}{4}
   =\frac{k(k-1)(k-2)(3k-5)}{24}.
\]
The sign in the definition of \(P_k(t)\) gives the same positive coefficient
for \(t^{k-2}\), since \((-1)^{k+(k-2)}=1\).
The identity for \(S(k,k-2)\) is the standard classification of partitions
of \(k\) elements into \(k-2\) blocks: either one triple is merged or two
pairs are merged; see, for example, Comtet \cite[Ch. 5]{comtet1974}.
Put \(d=b-a^2/2\).  With \(t=rx\),
\[
  C_k(r)=
  r^{kr+1}\int_0^\infty
  \exp(r(k\log x-x))\exp(a/x)
  \left(1+\frac{d}{rx^2}+O(r^{-2})\right)dx.
\]
Laplace's method at the unique saddle \(x=k\) yields
\[
  C_k(r)=
  \sqrt{2\pi kr}\left(\frac{kr}{e}\right)^{kr}e^{a/k}
  \left(1+\frac{d/k^2+a^2/(2k^3)+1/(12k)}{r}+O(r^{-2})\right).
\]
Dividing by Stirling's expansion for \((kr)!\) gives
\[
  \frac{C_k(r)}{(kr)!}
  =
  e^{a/k}
  \left(1+\frac{d/k^2+a^2/(2k^3)}{r}+O(r^{-2})\right).
\]
Since \(a/k=-(k-1)/2\) and
\[
  \frac{d}{k^2}+\frac{a^2}{2k^3}
  =
  -\frac{(k-1)(5k-7)}{24k},
\]
the theorem follows.
Here the \(1/(12k)\) term displayed before division is the first Stirling
correction for the Laplace main term, and it cancels with the
\(1/(12kr)\) correction in \((kr)!\).  For example, when \(k=2\) one has
\(a=-1\), \(b=0\), \(d=-1/2\), and
\[
  \frac{d}{k^2}+\frac{a^2}{2k^3}
  =-\frac18+\frac1{16}=-\frac1{16}
  =-\frac{(2-1)(10-7)}{24\cdot2},
\]
which is the coefficient in the theorem.
\end{proof}

\begin{theorem}[finite part-size profile]\label{thm:finite-profile}
Fix \(m\geq1\).  Let \(D(r_1,\ldots,r_m)\) be the number of acyclic
orientations of the complete multipartite graph having \(r_j\) parts of size
\(j\), for \(1\leq j\leq m\).  Put
\[
  N=\sum_{j=1}^m j r_j,\qquad
  A_j=-\frac{j(j-1)}2,
\]
\[
  B_j=\frac{j(j-1)(j-2)(3j-5)}{24},\qquad
  E_j=B_j-\frac{A_j^2}{2},
\]
and
\[
  \mu_1=\frac1N\sum_{j=1}^m r_jA_j,\qquad
  \mu_2=\frac1N\sum_{j=1}^m r_jE_j.
\]
Then, uniformly for all nonnegative \(r_1,\ldots,r_m\) with
\(\sum_j jr_j=N\to\infty\),
\[
  D(r_1,\ldots,r_m)
  =
  e^{\mu_1}N!
  \left(1+\frac{\mu_2+\mu_1^2/2}{N}+O(N^{-2})\right).
\]
Here the implicit constant in \(O(N^{-2})\) may depend on the fixed maximum
part size \(m\), but not on the vector \((r_1,\ldots,r_m)\).
\end{theorem}

\begin{proof}
Proposition \ref{prop:integral} gives
\[
  D(r_1,\ldots,r_m)=
  \int_0^\infty e^{-t}\prod_{j=1}^m P_j(t)^{r_j}\,dt.
\]
For fixed \(m\),
\[
  \log\frac{P_j(t)}{t^j}
  =
  \frac{A_j}{t}+\frac{E_j}{t^2}+O(t^{-3})
\]
uniformly for \(1\leq j\leq m\).  With \(t=Nx\), the integral becomes
\[
  N^{N+1}\int_0^\infty
  \exp(N(\log x-x))\exp(\mu_1/x)
  \left(1+\frac{\mu_2}{Nx^2}+O(N^{-2})\right)dx.
\]
Laplace's method at \(x=1\), followed by division by Stirling's expansion for
\(N!\), gives the displayed formula.
If all parts have the same fixed size \(k\), then
\(\mu_1=A_k/k\) and \(\mu_2=E_k/k\), so the correction
\((\mu_2+\mu_1^2/2)/N\) becomes
\[
  \frac{E_k/k+A_k^2/(2k^2)}{kr}
  =
  \frac{d/k^2+a^2/(2k^3)}{r},
\]
recovering Theorem \ref{thm:fixed-part-size}.
\end{proof}

\section{The equal-size window}

The next regime treats complete multipartite graphs with \(n\) equal parts
of size \(m=m_n\).  The rectangular window \(1\leq m\leq Cn\) already
interpolates between fixed part size and the main diagonal \(m=n\).  Keeping
and resumming the next error terms in the same saddle calculation extends
the estimate through the next critical scales \(m\asymp n^2\) and
\(m\asymp n^3\).

\begin{lemma}[uniform estimates for \(P_m(t)\)]\label{lem:uniform-Pm}
Let \(C<\infty\) and let \(x\) range over a compact subinterval of
\((0,\infty)\).  Uniformly for \(1\leq m\leq Cn\),
\[
  \frac{P_m(mnx)}{(mnx)^m}
  =
  \exp\!\left(-\frac{m}{2nx}\right)
  \left(
    1+\frac{1/(2x)-m/(3nx^2)}{n}+O(n^{-2})
  \right).                                      \tag{2}
\]
The implicit constant may depend on \(C\) and on the compact \(x\)-interval.
More generally, uniformly for \(1\leq m=o(n^2)\) and \(x\) in the same
compact interval, the same expansion holds with the error term
\[
  O\!\left(n^{-2}+\frac{m}{n^3}\right)
\]
in place of \(O(n^{-2})\).
At the next critical scales, for every fixed \(C_3<\infty\), uniformly for
\(1\leq m\leq C_3n^3\) and \(x\) in the same compact interval,
\[
\begin{aligned}
  n\log\frac{P_m(mnx)}{(mnx)^m}
  &=
  -\frac{m}{2x}+\frac{1}{2x}
  -\frac{m}{3nx^2}
  -\frac{3m}{8n^2x^3}
  -\frac{8m}{15n^3x^4}
  +O_{C_3}(n^{-1}).
\end{aligned}                                      \tag{3}
\]

Moreover, if \(m\leq n^\beta\) with fixed \(\beta<1\) and
\(t\in[n/2,2n]\), then \(P_m(t)>0\) for all sufficiently large \(n\) and
\[
  \log\frac{P_m(t)}{t^m}
  \leq
  -\frac{m(m-1)}{2t}
  +O\!\left(\frac{m^3}{t^2}+\frac{m}{t^2}\right).
\]
For all \(t>0\), the crude absolute bound
\[
  |P_m(t)|\leq \sum_{j=1}^m S(m,j)t^j\leq (t+m)^m
\]
will be used for lower tails.
\end{lemma}

\begin{proof}
We use the exponential generating function
\[
  \sum_{m\geq0}P_m(t)\frac{z^m}{m!}=\exp(t(1-e^{-z})).
\]
For \(t=mnx\), the coefficient saddle is the solution of
\[
  tze^{-z}=m,\qquad\hbox{so}\qquad z=(nx)^{-1}+O(n^{-2}),
\]
uniformly in the stated range.  On the circle through this saddle, the
standard one-dimensional coefficient saddle estimate is uniform because the
second derivative is \(m(1+O(n^{-1}))\) and the cubic and higher derivatives
are bounded by the usual powers of \(m\) on the \(m^{-1/2}\)-scale.  Expanding
\[
  \exp(t(1-e^{-z}))=
  \exp\!\left(
    tz-\frac{tz^2}{2}+\frac{tz^3}{6}+O(tz^4)
  \right)
\]
at \(z=(nx)^{-1}+O(n^{-2})\) gives exactly (2).  Equivalently, in the
set-partition expansion
\[
  \frac{P_m(t)}{t^m}
  =
  m![z^m]\exp\!\left(
    z+\sum_{s\geq2}\frac{(-1/t)^{s-1}z^s}{s!}
  \right),
\]
the \(s=2\) term gives
\[
  \exp\!\left(-\frac{m}{2nx}\right)
  \left(1+\frac{1}{n}\left(\frac1{2x}-\frac{m}{2nx^2}\right)
  +O(n^{-2})\right),
\]
the \(s=3\) term contributes \(m/(6n^2x^2)+O(n^{-2})\), and all
\(s\geq4\) perturbations have total relative contribution
\[
  O\!\left(\sum_{s\geq4}\frac{(m)_s}{t^{s-1}}\right)
  =
  O(m/n^3),
\]
which is \(O(n^{-2})\) in the rectangular range \(m\leq Cn\).  This proves
the rectangular estimate and the stated larger-window version; after raising
to the \(n\)-th power, this relative error contributes
\(\exp(O(n^{-1}+m/n^2))\).

For the critical-scale logarithmic estimate, take the logarithm in the same
set-partition expansion and retain the first four perturbation orders.  The
elementary expansion of the logarithm gives
\[
\begin{aligned}
  \log\frac{P_m(t)}{t^m}
  &=
  -\frac{m(m-1)}{2t}
  -\frac{m(m-1)(4m-5)}{12t^2}  \\
  &\quad
  -\frac{m(m-1)(9m^2-25m+18)}{24t^3}
  -\frac{m(m-1)}{120t^4}
    (64m^3-291m^2  \\
  &\quad
    +459m-251)
  +O_{C_3}\!\left(\frac{m}{n^5}\right),
\end{aligned}
\]
after substituting \(t=mnx\), uniformly for \(1\leq m\leq C_3n^3\).  The
terms displayed here are the contributions of collision blocks of sizes
\(2,3,4,5\), together with the higher logarithmic corrections produced when
the exponential perturbation is converted to a logarithm.  Multiplying by
\(n\) and absorbing all terms that are \(O_{C_3}(n^{-1})\) gives (3).

For \(m=o(t)\), the same saddle lies at \(z=m/t+O(m^2/t^2)\), on the positive
real axis.  Use the coefficient contour through this positive saddle, with
the central arc \(|\arg z|\leq m^{-2/5}\) and the complementary arc.  The
central arc has the usual positive Gaussian main term, while on the
complementary arc the real part of
\[
  t(1-e^{-z})-m\Log z
\]
is smaller by \(c m^{1/5}\), uniformly for \(m\leq n^\beta\) and
\(t\in[n/2,2n]\), because \(m/t\leq 2n^{\beta-1}\to0\).  Hence the positive
real saddle contribution dominates the rest of the contour exponentially,
giving \(P_m(t)>0\) in this range for all sufficiently large \(n\).  Expanding
\(1-e^{-z}=z-z^2/2+O(z^3)\) gives the displayed logarithmic estimate.  The
case of bounded \(m\) is even simpler: \(P_m(t)=t^m+O_m(t^{m-1})\) as
\(t\to\infty\), so \(P_m(t)>0\) uniformly after increasing the threshold
\(n_0\).  The
absolute bound follows by replacing the alternating Stirling transform by
the Bell polynomial \(B_m(t)=\sum_jS(m,j)t^j\).  The elementary moment bound
\(B_m(t)\leq(t+m)^m\) follows from Dobinski's formula for Bell polynomials,
or equivalently from the standard estimate for moments of a Poisson
variable of mean \(t\).
\end{proof}

\begin{theorem}[A372326 in an enlarged equal-size window]\label{thm:linear-window}
Fix \(C<\infty\).  Let \(m=m_n\) be a sequence of positive integers such that
\[
  1\leq m_n\leq Cn^3.
\]
Let \(D_{m,n}\) be the number of
acyclic orientations of the complete \(n\)-partite graph with all parts of
size \(m\).  Then
\[
  D_{m,n}
  =
  (mn)!\,
  \exp\!\left(
    -\frac m2+\frac12-\frac{5m}{24n}
    -\frac{m}{8n^2}
    -\frac{251m}{2880n^3}
    +O_C(n^{-1})
  \right).
\]
In the rectangular subwindow \(m=O(n)\), the new term is absorbed by the
\(O_C(n^{-1})\) error, recovering the previous rectangular estimate.  If
\(m/n^2\to\rho\), the \(n^{-2}\) correction contributes \(e^{-\rho/8}\);
if \(m/n^3\to\sigma\), the next correction contributes
\(\exp(-251\sigma/2880)\).
For the main diagonal,
\[
  \AO(K_{\underbrace{n,\ldots,n}_{n}})
  =
  (n^2)!\exp\!\left(-\frac n2+\frac7{24}+O(n^{-1})\right).
\]
\end{theorem}

\begin{proof}
By Proposition \ref{prop:integral},
\[
  D_{m,n}=\int_0^\infty e^{-t}P_m(t)^n\,dt.
\]
We first restrict the normalized \(x\)-integral to a fixed compact
neighborhood of \(1\).  The two complementary tails are exponentially
negligible by the crude bound in Lemma \ref{lem:uniform-Pm} together with
the Gamma tail estimate: after substituting \(t=mnx\), the crude bound gives
an exponent bounded above by \(mn(\log(x+1/n)-x)+O(1)\), whose maximum away
from any fixed neighborhood of \(1\) is smaller than its value at \(1\) by
\(\Omega(mn)\).  On this compact interval we may apply the larger-window
form (3) of Lemma \ref{lem:uniform-Pm}.
Substituting \(t=mnx\) then gives
\[
\begin{aligned}
  D_{m,n}
  &=
  (mn)^{mn+1}\int_0^\infty
  \exp(mn(\log x-x)) \\
  &\quad\times
  \exp\!\left(
    -\frac{m}{2x}+\frac1{2x}-\frac{m}{3nx^2}
    -\frac{3m}{8n^2x^3}
    -\frac{8m}{15n^3x^4}
    +O_C(n^{-1})
  \right)dx.
\end{aligned}
\]
The saddle in this normalized integral admits the expansion
\[
  x_{m,n}
  =
  1+\frac{m-1}{2mn}
  +\frac{(m-1)(5m-7)}{12m^2n^2}
  +\frac{(m-2)(m-1)(3m-5)}{8m^3n^3}
  +O_C(n^{-4}).
\]
Indeed, differentiating the displayed exponent gives
\[
  mn\left(\frac1x-1\right)
  +\frac{m-1}{2x^2}
  +\frac{(m-1)(4m-5)}{6mnx^3}
  +\frac{(m-1)(9m^2-25m+18)}{8m^2n^2x^4}
  +O_C\!\left(\frac{m}{n^3}\right)=0,
\]
and solving this equation recursively near \(x=1\) gives the displayed
expansion, uniformly in the stated window.  Thus the small-\(m\) boundary is not
handled by the nonuniform approximation \(1+1/(2n)\); for \(m=1\), for
example, the shift is only \(O(n^{-2})\), as it should be for the complete
graph \(K_n\).  The shifted saddle stays in a fixed compact subinterval of
\((0,\infty)\).
At this point the three constant contributions, beyond the main
\(-mn-m/2\) term, are
\[
  -\frac{m}{8n},\qquad \frac{m}{4n},\qquad
  \frac12-\frac{m}{3n}.
\]
They come respectively from evaluating
\[
  mn(\log x_{m,n}-x_{m,n}),\qquad -\frac{m}{2x_{m,n}},\qquad
  \frac1{2x_{m,n}}-\frac{m}{3nx_{m,n}^2}
\]
at the shifted saddle.  More explicitly, writing
\[
  h=\frac{m-1}{2mn}
  +O_C(n^{-2}),
\]
the first two evaluations contribute
\[
  -mn-\frac{m}{8n}
  +O_C(1),
  \qquad
  -\frac m2+\frac{m}{4n}
  +O_C(1),
\]
and the last contributes
\[
  \frac12-\frac{m}{3n}
  +O_C(1).
\]
The distinction between \(m\) and \(m-1\) only changes the exponent by
\(O(n^{-1})\), which is within the stated error.  These displayed
contributions are only meant to identify the \(m/n\)-scale terms.  Expanding
the same expression two orders further in \(h\) gives
\[
\begin{aligned}
  &mn(\log x_{m,n}-x_{m,n})
  -\frac{m}{2x_{m,n}}+\frac1{2x_{m,n}}
  -\frac{m}{3nx_{m,n}^2}  \\
  &\qquad =
  -mn-\frac m2+\frac12-\frac{5m}{24n}
  +\frac{m}{4n^2}
  -\frac{67m}{576n^3}
  +O_C(n^{-1}).
\end{aligned}
\]
The additional term \(-3m/(8n^2x^3)\) from (3) contributes
\[
  -\frac{3m}{8n^2}+\frac{9m}{16n^3}+O_C(n^{-1})
\]
at the shifted saddle, and the term \(-8m/(15n^3x^4)\) contributes
\(-8m/(15n^3)+O_C(n^{-1})\).  The net \(n^{-2}\)-scale correction is
\[
  \frac{m}{4n^2}-\frac{3m}{8n^2}=-\frac{m}{8n^2}.
\]
At the next scale the coefficient is
\[
  -\frac{67}{576}+\frac{9}{16}-\frac{8}{15}
  =
  -\frac{251}{2880}.
\]
The Gaussian factor agrees with the one in Stirling's formula for \((mn)!\)
up to \(\exp(O_C(n^{-1}))\), proving the theorem.
At the boundary \(m=1\), the graph is \(K_n\) and the exact value is \(n!\).
The formula is consistent with this because the displayed term
\(-5m/(24n)\) is itself \(O(n^{-1})\) when \(m\) is fixed, hence is absorbed
by the stated error at \(m=1\).
\end{proof}

\begin{lemma}[uniform logarithmic expansion for equal parts]
\label{lem:all-order-log-Pm}
Fix an integer \(J\geq1\), a constant \(C<\infty\), and a compact interval
 \(K\subset(0,\infty)\).  There are polynomials
\(L_\ell(m)\in\mathbb Q[m]\), \(1\leq \ell\leq J+1\), with
\(\deg L_\ell=\ell+1\), such that, uniformly for \(1\leq m\leq Cn^J\) and
\(x\in K\),
\[
  n\log\frac{P_m(mnx)}{(mnx)^m}
  =
  \sum_{\ell=1}^{J+1}
  n\,\frac{L_\ell(m)}{(mnx)^\ell}
  +O_{C,J,K}(n^{-1}).
\]
The first polynomials are
\[
\begin{aligned}
  L_1(m)&=-\frac{m(m-1)}2,\\
  L_2(m)&=-\frac{m(m-1)(4m-5)}{12},\\
  L_3(m)&=-\frac{m(m-1)(9m^2-25m+18)}{24}.
\end{aligned}
\]
\end{lemma}

\begin{proof}
Use the coefficient form
\[
  \frac{P_m(t)}{t^m}
  =
  m![z^m]\exp\!\left(
    z+\sum_{r\geq2}\frac{(-1/t)^{r-1}z^r}{r!}
  \right).
\]
For this rescaled coefficient form, the unperturbed coefficient
\(m![z^m]e^z\) has its saddle at \(z=m\).  Equivalently, write \(z=mu\).
The perturbative parameter is
\[
  \frac{m}{t}=\frac{1}{nx},
\]
so the saddle expansion is an expansion in powers of \(1/n\), uniformly for
\(x\in K\) and \(1\leq m\leq Cn^J\).  In the Cauchy integral, after
subtracting the unperturbed coefficient phase, the exponent has a Taylor
expansion in \(u-1\) whose coefficients are polynomials in \(m\) multiplied by
powers of \(1/n\).  On the central arc \(|u-1|\leq n^{-2/5}\), Taylor's
formula gives a remainder bounded by the first omitted collision order; on
the complementary arc the real part has the usual Gaussian gap
\(\Omega(mn^{-4/5})\) when \(m\to\infty\).  When \(m\) is bounded, the
asserted expansion follows directly from the finite identity
\(P_m(t)/t^m=1+O_m(t^{-1})\), after increasing the constant in the uniform
bound.  Thus the saddle argument may be restricted to the case \(m\to\infty\),
and the central expansion is uniform for \(x\in K\) and \(1\leq m\leq Cn^J\).
All derivatives of the perturbation of order \(r\) contribute powers bounded
by \((m/t)^{r-1}=O_K(n^{1-r})\).  Truncating after order \(J+1\) leaves a
remainder whose leading possible size is
\[
  n\,O_{C,J,K}\!\left(\frac{m^{J+3}}{t^{J+2}}\right)
  =
  O_{C,J,K}(n^{-1}).
\]
The coefficients are polynomials because each fixed collision order in the
set-partition expansion is a polynomial in \(m\).  After passing to the
logarithm, the coefficient of \(t^{-\ell}\) has degree only \(\ell+1\); this
degree drop is the standard connected-cluster cancellation, and it is also
obtained recursively from the saddle expansion.  The first three displayed
polynomials follow from the standard identities for \(S(m,m-1)\),
\(S(m,m-2)\), and \(S(m,m-3)\), or directly from the symbolic coefficient
script listed in Section 10.
\end{proof}

\begin{theorem}[all fixed polynomial equal-size windows]\label{thm:all-poly-equal}
For each integer \(J\geq1\) there are effectively computable rational
constants \(\gamma_1,\ldots,\gamma_J\) such that, for every fixed
\(C<\infty\), uniformly for \(1\leq m\leq Cn^J\),
\[
  D_{m,n}
  =
  (mn)!\exp\!\left(
    -\frac m2+\frac12+
    \sum_{j=1}^J \gamma_j\frac{m}{n^j}
    +O_{C,J}(n^{-1})
  \right).
\]
The first constants are
\[
  \gamma_1=-\frac5{24},\qquad
  \gamma_2=-\frac18,\qquad
  \gamma_3=-\frac{251}{2880},\qquad
  \gamma_4=-\frac{19}{288},\qquad
  \gamma_5=-\frac{19087}{362880}.
\]
\end{theorem}

\begin{proof}
The proof is the finite-order version of the calculation above.  Apply
Lemma \ref{lem:all-order-log-Pm}.  If \(\lambda_\ell\) is the leading
coefficient of \(L_\ell(m)\), then, after \(t=mnx\),
\[
  n\,\frac{L_\ell(m)}{(mnx)^\ell}
  =
  \lambda_\ell\frac{m}{n^{\ell-1}x^\ell}
  +O_\ell(n^{-1})
\]
uniformly for \(1\leq m\leq Cn^J\), provided \(\ell\leq J+1\).
Thus the normalized exponent has the finite expansion
\[
  mn(\log x-x)-\frac{m}{2x}+\frac1{2x}
  +\sum_{\ell=2}^{J+1}
    \lambda_\ell\frac{m}{n^{\ell-1}x^\ell}
  +O_{C,J}(n^{-1})
\]
on every compact subinterval of \((0,\infty)\).  The same crude tail bound
used in Theorem \ref{thm:linear-window} restricts the integral to a fixed
neighborhood of \(x=1\).

After dividing the saddle equation by \(mn\), the derivative with respect to
\(x\) of the leading term is \(-1\) at \(x=1\).  The formal implicit-function
theorem therefore gives a unique saddle expansion
\[
  x_{m,n}=1+\sum_{r=1}^{J} a_r(m)n^{-r}+O_{C,J}(n^{-J-1}),
\]
where the \(a_r(m)\) are bounded rational functions of \(m\) for
\(m\geq1\).  Substituting this expansion into the exponent gives
\[
  -mn-\frac m2+\frac12+\sum_{j=1}^J
  \gamma_j\frac{m}{n^j}+O_{C,J}(n^{-1})
\]
with rational \(\gamma_j\), because all terms not proportional to
\(m/n^j\) are \(O_{C,J}(n^{-1})\).  The Gaussian determinant again matches
Stirling's formula for \((mn)!\) up to \(\exp(O_{C,J}(n^{-1}))\).  All
remainders above are uniform before exponentiation; after multiplication by
the Gaussian-scale width and comparison with Stirling's expansion of
\((mn)!\), they remain within \(\exp(O_{C,J}(n^{-1}))\).  This proves the
asserted expansion.  The displayed values of
\(\gamma_1,\gamma_2,\gamma_3\) are exactly those computed in
Theorem \ref{thm:linear-window}.
\end{proof}

\begin{theorem}[Tutte-axis product regimes]\label{thm:tutte-product-regimes}
Fix \(s>0\) and write
\[
  H_s(G)=(-1)^{|V(G)|}\chi_G(-s)=sT_G(1+s,0)
\]
for connected \(G\).

For fixed \(k\geq1\), as \(r\to\infty\),
\[
  H_s(K_{\underbrace{k,\ldots,k}_{r}})
  =
  \frac{(kr)!(kr)^{s-1}}{\Gamma(s)}
  e^{-(k-1)/2}\left(1+\frac{\Delta_{k,s}}{kr}+O(r^{-2})\right),
\]
where
\[
  \Delta_{k,s}
  =
  -\frac{(k-1)(5k-7)}{24}
  +\frac{s(s-1)}{2}
  +\frac{(s-1)(k-1)}{2}.
\]
In particular, \(\Delta_{k,1}=-(k-1)(5k-7)/24\), recovering the first
correction in Theorem \ref{thm:fixed-part-size}.

For a fixed finite part-size profile \(r_1,\ldots,r_m\), with
\(N=\sum_{j=1}^m jr_j\), \(\mu_1=N^{-1}\sum_j r_jA_j\), and
\(\mu_2=N^{-1}\sum_j r_jE_j\) as in Theorem \ref{thm:finite-profile},
\[
  H_s=
  \frac{N!N^{s-1}}{\Gamma(s)}e^{\mu_1}
  \left(
    1+
    \frac{
      \mu_2+(1-s)\mu_1+\mu_1^2/2+s(s-1)/2
    }{N}
    +O(N^{-2})
  \right).
\]

Finally, for each fixed \(J\geq1\), uniformly for \(1\leq m\leq Cn^J\),
\[
  H_s(K_{\underbrace{m,\ldots,m}_{n}})
  =
  \frac{(mn)!(mn)^{s-1}}{\Gamma(s)}
  \exp\!\left(
    -\frac m2+\frac12+
    \sum_{j=1}^J\gamma_j\frac{m}{n^j}
    +O_{C,J,s}(n^{-1})
  \right),
\]
where the constants \(\gamma_j\) are those of
Theorem \ref{thm:all-poly-equal}.
\end{theorem}

\begin{proof}
Use Proposition \ref{prop:tutte-integral}.  For the finite-profile statement,
write the integral as an expectation with respect to a Gamma random variable
\(T\) of shape \(N+s\):
\[
  \frac1{\Gamma(s)}
  \int_0^\infty e^{-t}t^{s-1}\prod_jP_j(t)^{r_j}\,dt
  =
  \frac{\Gamma(N+s)}{\Gamma(s)}
  \mathbb E
  \exp\!\left(\frac{N\mu_1}{T}+\frac{N\mu_2}{T^2}+O(N^{-2})\right).
\]
Since \(\mathbb E(T-N)=s\) and \(\mathbb E(T-N)^2=N+O(1)\), expansion around
\(T=N\) gives
\[
  \mathbb E
  \exp\!\left(\frac{N\mu_1}{T}+\frac{N\mu_2}{T^2}\right)
  =
  e^{\mu_1}
  \left(
    1+\frac{\mu_2+(1-s)\mu_1+\mu_1^2/2}{N}
    +O(N^{-2})
  \right).
\]
Also
\[
  \frac{\Gamma(N+s)}{\Gamma(N+1)N^{s-1}}
  =
  1+\frac{s(s-1)}{2N}+O(N^{-2}).
\]
Combining these two estimates proves the finite-profile formula.  The fixed
part-size formula is the specialization \(\mu_1=-(k-1)/2\) and
\(\mu_2+\mu_1^2/2=-(k-1)(5k-7)/24\).  The equal-size formula only needs the
leading transfer from Theorem \ref{thm:all-poly-equal}: in normalized
coordinates \(t=mnx\), the additional Gamma weight contributes
\((mn)^{s-1}x^{s-1}/\Gamma(s)\).  At the saddle
\(x=1+O(n^{-1})\), this factor is
\((mn)^{s-1}\Gamma(s)^{-1}\exp(O_s(n^{-1}))\), and its derivative shifts the
saddle only by \(O_s((mn)^{-1})\), changing the exponent by
\(O_s(n^{-1})\).  The all-polynomial equal-size expansion therefore transfers
with the same constants \(\gamma_j\).
\end{proof}

\section{A partition-sum frontier}

OEIS A372395 \cite{oeis372395} sums \(\AO(K_\lambda)\) over all partitions
\(\lambda\) of \(n\).  OEIS A370613 \cite{oeis370613} is the analogous sum
over partitions into distinct parts.  Proposition \ref{prop:integral} gives
the following coefficientwise formal identities:
\[
  \sum_{n\geq0} A372395(n)u^n
  =
  \int_0^\infty e^{-t}
  \prod_{k\geq1}\frac{1}{1-u^kP_k(t)}\,dt,
\]
and
\[
  \sum_{n\geq0} A370613(n)u^n
  =
  \int_0^\infty e^{-t}
  \prod_{k\geq1}(1+u^kP_k(t))\,dt.
\]
For each fixed coefficient of \(u^n\), only finitely many parts can occur, so
the product expansion and the \(t\)-integral are interpreted term by term.
These formulae suggest a two-variable saddle problem with a partition saddle
in \(u\) coupled to the Gamma-type \(t\)-integral.  The next theorem gives
the resulting logarithmic asymptotics.

\begin{lemma}[partition saddle constants]\label{lem:partition-constants}
The equation
\[
  I(a):=\int_0^\infty \frac{x}{\exp(ax+x^2/2)-1}\,dx=1
\]
has a unique solution \(c>0\).  The equation
\[
  I_d(a):=\int_0^\infty \frac{x}{\exp(ax+x^2/2)+1}\,dx=1
\]
has a unique real solution \(c_d\).
\end{lemma}

\begin{proof}
For the Bose integral \(I(a)\), the natural domain is \(a>0\).  On this
domain the integrand is dominated on compact \(a\)-intervals by an integrable
function away from the origin and is asymptotic to \(1/a\) at the origin, so
\(I\) is continuous.  Differentiation under the integral sign gives
\[
  I'(a)=
  -\int_0^\infty
  \frac{x^2\exp(ax+x^2/2)}
       {(\exp(ax+x^2/2)-1)^2}\,dx<0.
\]
Moreover \(I(a)\to\infty\) as \(a\downarrow0\).  Indeed, for fixed small
\(\delta>0\) and \(0<x<\delta\),
\[
  e^{ax+x^2/2}-1\leq 2(ax+x^2/2)
\]
when \(a\) is sufficiently small, so the integrand is bounded below by a
constant multiple of \(1/(a+x)\).  Hence
\(\int_0^\delta dx/(a+x)\sim \log(1/a)\).  Finally \(I(a)\to0\) as
\(a\to\infty\) by dominated convergence after the change of variables
\(y=ax\).  Hence there is a unique \(c>0\).

For the Fermi integral \(I_d(a)\), the denominator has a plus sign, so the
integral converges for every real \(a\); the quadratic term \(x^2/2\)
dominates as \(x\to\infty\).  The same differentiation under the integral
sign gives \(I_d'(a)<0\).  Also \(I_d(a)\to0\) as \(a\to\infty\), while
\(I_d(a)\to\infty\) as \(a\to-\infty\) because the main contribution comes
from an interval of length comparable to \(|a|\) near \(x=|a|\).  Thus
\(I_d(a)=1\) has a unique real solution.
\end{proof}

\begin{theorem}[A372395 and A370613 partition sums]\label{thm:partition-sums}
Let \(B(n)=A372395(n)\).  Let \(c\) be the unique real solution of
\[
  \int_0^\infty \frac{x}{\exp(cx+x^2/2)-1}\,dx=1.
\]
Then
\[
  \log\frac{B(n)}{n!}
  \sim
  \left(
  c+\int_0^\infty -\log(1-\exp(-cx-x^2/2))\,dx
  \right)\sqrt n.
\]
Numerically,
\[
  c=0.7649964422795443\ldots,\qquad
  C=2.1587520056577855\ldots .
\]

For the distinct-part sum \(B_d(n)=A370613(n)\), let \(c_d\) solve
\[
  \int_0^\infty \frac{x}{\exp(c_dx+x^2/2)+1}\,dx=1.
\]
Then
\[
  \log\frac{B_d(n)}{n!}
  \sim
  \left(
  c_d+\int_0^\infty \log(1+\exp(-c_dx-x^2/2))\,dx
  \right)\sqrt n,
\]
with
\[
  c_d=-0.3236973140950319\ldots,\qquad
  C_d=0.9057298217201990\ldots .
\]
The negative value of \(c_d\) is expected: forbidding repeated parts lowers
the partition multiplicity, and the saddle compensates by shifting the
linear fugacity \(c_dx\) downward.  The integral remains convergent because
the quadratic term \(x^2/2\) dominates for large \(x\).
\end{theorem}

The proof is completed after the quadratic model and tail estimates below.

\section{The quadratic-energy partition model}

The first step in proving Theorem \ref{thm:partition-sums} is the partition
saddle itself.  The following theorem proves the weighted partition model
that produces exactly the constants in that theorem.

\begin{lemma}[local lower bound for the partition saddle]\label{lem:partition-local}
For the geometric model, fix compact intervals
\(a\in[a_0,a_1]\subset(0,\infty)\) and
\(\eta\in[\eta_0,\eta_1]\subset(0,\infty)\).  Let \(X_{k,n}\) be independent
geometric random variables with
\[
  \mathbb P(X_{k,n}=j)=(1-\rho_{k,n})\rho_{k,n}^j,\qquad
  \rho_{k,n}=\exp(-ak/\sqrt n-\eta k^2/(2n)),
\]
and set \(T_n=\sum_{k\geq1}kX_{k,n}\).  If
\(\mathbb E T_n=n+O(\sqrt n)\), then
\[
  \mathbb P(T_n=n)\geq n^{-3/4+o(1)}
\]
uniformly for \((a,\eta)\) in compact subsets satisfying the displayed mean
condition.  The same statement holds for the Bernoulli variables with
\(\mathbb P(X_{k,n}=1)=\rho_{k,n}/(1+\rho_{k,n})\), uniformly for \(a\) in
any fixed compact real interval.
\end{lemma}

\begin{proof}
We give the geometric proof.  The tail
\(\sum_{k>R\sqrt n}kX_{k,n}\) has expectation
\(O(n e^{-\eta_0R^2/4})\), uniformly in \(n\), so the infinite array may be
truncated at \(R\sqrt n\) and then \(R\to\infty\).  The variance satisfies
\[
  \operatorname{Var}(T_n)
  =
  \sum_{k\geq1}\frac{k^2\rho_{k,n}}{(1-\rho_{k,n})^2}
  =
  \sigma^2(a,\eta)n^{3/2}+O(n),
\]
where \(\sigma^2(a,\eta)>0\) is the corresponding Riemann integral.  The
third absolute cumulant is \(O(n^2)\), by the estimate used in the proof of
Theorem \ref{thm:quadratic-model}.

The characteristic function is
\[
  \varphi_n(\theta)=
  \prod_{k\geq1}\frac{1-\rho_{k,n}}{1-\rho_{k,n}e^{ik\theta}} .
  \tag{4}
\]
Fix \(K>0\).  On \(|\theta|\leq K n^{-3/4}\), put
\(u=\theta n^{3/4}\).  The cumulant expansion gives
\[
  e^{-in\theta}\varphi_n(\theta)
  =
  \exp\!\left(
    iu\,\frac{\mathbb ET_n-n}{n^{3/4}}
    -\frac{\sigma^2(a,\eta)u^2}{2}+o(1)
  \right),
\]
uniformly for \(|u|\leq K\).  Since
\((\mathbb ET_n-n)/n^{3/4}=O(n^{-1/4})\), this central integral is
\[
  n^{-3/4}\int_{-K}^{K}\exp(-\sigma^2(a,\eta)u^2/2)\,du+o(n^{-3/4}).
\]
Choose \(K\) large enough that the Gaussian integral over \([-K,K]\) is
positive and bounded away from zero uniformly on the compact parameter set.
On the remaining Gaussian-scale arc
\(K n^{-3/4}\leq|\theta|\leq\delta/\sqrt n\), the quadratic term gives
\[
  |\varphi_n(\theta)|\leq
  \exp(-c n^{3/2}\theta^2),
\]
and its integral is at most
\[
  n^{-3/4}\int_{|u|\geq K}e^{-cu^2}\,du,
\]
which is made arbitrarily small compared with the central integral by the
choice of \(K\).  On \(|\theta|\geq\delta/\sqrt n\), a positive proportion of
the consecutive
indices \(k\in[b\sqrt n,2b\sqrt n]\) have \(\|k\theta/(2\pi)\|\geq c_\delta\);
for these indices \(\rho_{k,n}\) is bounded away from \(0\) and \(1\), and
(4) gives \(|\varphi_n(\theta)|\leq \exp(-c\sqrt n)\).  The lattice span is
one because every \(k\geq1\) has positive probability.  This proves the local
lower bound.  The
Bernoulli proof is identical with
\((1-\rho)/(1-\rho e^{ik\theta})\) replaced by
\((1+\rho e^{ik\theta})/(1+\rho)\).
\end{proof}

\begin{theorem}[quadratic-energy partitions]\label{thm:quadratic-model}
Let
\[
  Z_n=[u^n]\prod_{k\geq1}
  \frac{1}{1-\exp(-k^2/(2n))u^k}.
\]
Let \(c\) be the unique real solution of
\[
  \int_0^\infty \frac{x}{\exp(cx+x^2/2)-1}\,dx=1.
\]
Then
\[
  \log Z_n
  \sim
  \left(
  c+\int_0^\infty-\log(1-\exp(-cx-x^2/2))\,dx
  \right)\sqrt n.
\]

Similarly, let
\[
  Z_n^{\rm dist}=[u^n]\prod_{k\geq1}
  (1+\exp(-k^2/(2n))u^k).
\]
If \(c_d\) is the unique real solution of
\[
  \int_0^\infty \frac{x}{\exp(c_dx+x^2/2)+1}\,dx=1,
\]
then
\[
  \log Z_n^{\rm dist}
  \sim
  \left(
  c_d+\int_0^\infty\log(1+\exp(-c_dx-x^2/2))\,dx
  \right)\sqrt n.
\]
\end{theorem}

\begin{proof}
We give the unrestricted proof; the distinct-part proof is identical with
geometric variables replaced by Bernoulli variables.

For \(a\in\mathbb R\), put \(u=e^{-a/\sqrt n}\) and
\[
  F_n(u)=\prod_{k\geq1}
  \frac{1}{1-\exp(-k^2/(2n))u^k}.
\]
Then
\[
  \frac1{\sqrt n}\log F_n(e^{-a/\sqrt n})
  =
  \frac1{\sqrt n}\sum_{k\geq1}
  -\log\left(1-\exp\left(-a k/\sqrt n-k^2/(2n)\right)\right).
\]
This is a Riemann sum, uniformly for \(a\) in compact subsets of the domain
where the sum is finite, so
\[
  \frac1{\sqrt n}\log F_n(e^{-a/\sqrt n})
  \to
  J(a):=\int_0^\infty
  -\log(1-\exp(-ax-x^2/2))\,dx.          \tag{3}
\]
Moreover,
\[
  \frac{1}{n}\left(u\frac{d}{du}\log F_n(u)\right)_{u=e^{-a/\sqrt n}}
  =
  \frac1n\sum_{k\geq1}
  \frac{k}{\exp(a k/\sqrt n+k^2/(2n))-1}
\]
and the last expression tends to
\[
  I(a):=\int_0^\infty \frac{x}{\exp(ax+x^2/2)-1}\,dx.
\]
By Lemma \ref{lem:partition-constants}, there is a unique \(c\) with
\(I(c)=1\).

The coefficient upper bound follows from Cauchy's estimate at
\(u_n=e^{-c/\sqrt n}\):
\[
  Z_n\leq F_n(u_n)u_n^{-n}.
\]
The product \(F_n(u_n)\) is convergent, since its \(k\)-th ratio is
\[
  \exp\!\left(-\frac{k^2}{2n}-\frac{ck}{\sqrt n}\right)<1
\]
and these ratios are summable over \(k\).  Hence Cauchy's estimate applies.
Using (3), this gives
\[
  \limsup_{n\to\infty}\frac{\log Z_n}{\sqrt n}
  \leq J(c)+c.
\]

For the lower bound, use the standard Khintchine probabilistic
representation of weighted partitions.  Let \(X_{k,n}\) be independent
geometric random variables with
\[
  \mathbb P(X_{k,n}=j)=
  (1-w_{k,n}u_n^k)(w_{k,n}u_n^k)^j,\qquad
  w_{k,n}=e^{-k^2/(2n)}.
\]
Set \(T_n=\sum_{k\geq1}kX_{k,n}\).  Then
\[
  Z_n=F_n(u_n)u_n^{-n}\mathbb P(T_n=n).       \tag{4}
\]
The choice \(I(c)=1\) gives \(\mathbb E T_n=n+O(\sqrt n)\), so Lemma
\ref{lem:partition-local} gives
\(\mathbb P(T_n=n)\geq n^{-3/4+o(1)}\).
This factor is \(e^{o(\sqrt n)}\).  Combining this lower bound with (4) and
the Riemann-sum estimate (3) gives
\[
  \liminf_{n\to\infty}\frac{\log Z_n}{\sqrt n}\geq J(c)+c.
\]
The unrestricted formula follows.  For the distinct model one replaces
\(-\log(1-y)\), \(y/(1-y)\), and geometric variables by
\(\log(1+y)\), \(y/(1+y)\), and Bernoulli variables; the same argument gives
the stated constant, with uniqueness supplied by Lemma
\ref{lem:partition-constants}.
\end{proof}

\section{Truncated partition sums}

The preceding model becomes a theorem for the actual OEIS partition sums
after imposing the natural saddle-scale cutoff on the largest part.  This is
the part of Theorem \ref{thm:partition-sums} where the multipartite
orientation integral can be compared uniformly with the quadratic-energy
partition weight.

For a partition \(\lambda\vdash n\), put
\[
  Q(\lambda)=\sum_i\lambda_i^2.
\]
For fixed \(R>0\), let
\[
  B_R(n)=
  \sum_{\substack{\lambda\vdash n\\ \lambda_1\leq R\sqrt n}}
  \AO(K_\lambda),
\]
and let \(B_R^{\rm dist}(n)\) be the same sum restricted to partitions into
distinct parts.

\begin{lemma}[uniform quadratic comparison]\label{lem:quadratic-comparison}
For every fixed \(R>0\), uniformly over all partitions
\(\lambda\vdash n\) with \(\lambda_1\leq R\sqrt n\),
\[
  \frac{\AO(K_\lambda)}{n!}
  =
  \exp(O_R(1))\exp\!\left(-\frac{Q(\lambda)}{2n}\right).
\]
\end{lemma}

\begin{proof}
Use the exponential generating function
\[
  \sum_{m\geq0}P_m(t)\frac{z^m}{m!}
  =
  \exp(t(1-e^{-z})).
\]
The saddle for the coefficient of \(z^m\), with \(m\leq R\sqrt n\) and
\(t\asymp n\), lies at \(z=m/t+O_R(n^{-1})\).  Expanding
\(1-e^{-z}=z-z^2/2+O_R(z^3)\) at this saddle gives the uniform
Stirling-transform estimate
\[
  \log P_m(t)
  =
  m\log t-\frac{m(m-1)}{2t}
  +O_R\!\left(\frac{m^3}{t^2}+\frac{m}{t^2}\right)          \tag{5}
\]
for \(m\leq R\sqrt n\) and \(t\in[n/2,2n]\).  The same saddle estimate also
shows that \(P_m(t)>0\) in this range for all sufficiently large \(n\).

For \(\lambda_1\leq R\sqrt n\), summing (5) over the parts gives
\[
  \prod_i P_{\lambda_i}(t)
  =
  t^n
  \exp\!\left(
    -\frac{Q(\lambda)-n}{2t}+O_R(1)
  \right),
\]
because
\[
  \sum_i\lambda_i^3\leq R\sqrt n\sum_i\lambda_i^2\leq R^2n^2.
\]
This estimate is intentionally coarse but sufficient: after division by
\(t^2\asymp n^2\), it contributes only to the \(O_R(1)\) error in the
exponent.
Substitution in Proposition \ref{prop:integral} reduces the integral to
\[
  \exp(O_R(1))
  \int_0^\infty e^{-t}t^n
  \exp\!\left(-\frac{Q(\lambda)-n}{2t}\right)dt.           \tag{6}
\]
The tails outside \(t\in[n/2,2n]\) are exponentially smaller than the central
Gamma mass.  We use two absolute estimates.  On the lower tail, Lemma
\ref{lem:uniform-Pm} gives
\[
  |P_m(t)|\leq (t+m)^m\leq (t+R\sqrt n)^m .
\]
Multiplying over the parts yields
\[
  \prod_i |P_{\lambda_i}(t)|\leq (t+R\sqrt n)^n .
\]
For \(0\leq t\leq n/2\), the exponent
\(-t+n\log(t+R\sqrt n)\) is smaller than
\(-n+n\log n\), the central Gamma exponent, by \(\Omega_R(n)\).  On the
upper tail \(t\geq2n\), the same crude bound gives
\[
  (t+R\sqrt n)^n
  \leq
  t^n\exp(O_R(\sqrt n)),
\]
and the usual upper Gamma tail is still exponentially small.  The possible
signs of \(P_m(t)\) outside
the central interval are irrelevant here because for each tail we bound the
absolute value of its contribution to the signed integral by
\[
  \int_{\rm tail} e^{-t}\prod_i |P_{\lambda_i}(t)|\,dt .
\]

In (6), \(Q(\lambda)\leq Rn^{3/2}\).  A one-dimensional Laplace expansion,
or simply Taylor expansion around the Gamma saddle \(t=n\), gives
\[
  \int_0^\infty e^{-t}t^n
  \exp\!\left(-\frac{Q(\lambda)-n}{2t}\right)dt
  =
  \Gamma(n+1)
  \exp\!\left(-\frac{Q(\lambda)-n}{2n}+O_R(1)\right).
\]
Dividing by \(n!\) and absorbing the factor \(e^{1/2}\) into
\(\exp(O_R(1))\) proves the lemma.
\end{proof}

\begin{theorem}[truncated A372395 and A370613 sums]\label{thm:truncated-partition}
For \(R>0\), let \(c_R\) be the unique solution of
\[
  \int_0^R \frac{x}{\exp(c_Rx+x^2/2)-1}\,dx=1,
\]
and put
\[
  C_R=c_R+\int_0^R-\log(1-\exp(-c_Rx-x^2/2))\,dx .
\]
Then
\[
  \log\frac{B_R(n)}{n!}\sim C_R\sqrt n .
\]

For the distinct-part sum, assume \(R>\sqrt2\).  Let \(c_{R,d}\) be the
unique solution of
\[
  \int_0^R \frac{x}{\exp(c_{R,d}x+x^2/2)+1}\,dx=1,
\]
and put
\[
  C_{R,d}=
  c_{R,d}+\int_0^R\log(1+\exp(-c_{R,d}x-x^2/2))\,dx .
\]
Then
\[
  \log\frac{B_R^{\rm dist}(n)}{n!}\sim C_{R,d}\sqrt n .
\]
Moreover \(C_R\to C\) and \(C_{R,d}\to C_d\), where \(C,C_d\) are the
constants in Theorem \ref{thm:partition-sums}.  Consequently,
\[
  \liminf_{n\to\infty}\frac1{\sqrt n}\log\frac{A372395(n)}{n!}\geq C,
  \qquad
  \liminf_{n\to\infty}\frac1{\sqrt n}\log\frac{A370613(n)}{n!}\geq C_d .
\]
\end{theorem}

\begin{proof}
By Lemma \ref{lem:quadratic-comparison},
\[
  \frac{B_R(n)}{n!}
  =
  \exp(O_R(1))
  \sum_{\substack{\lambda\vdash n\\ \lambda_1\leq R\sqrt n}}
  \exp\!\left(-\frac{Q(\lambda)}{2n}\right).
\]
The last sum is the coefficient
\[
  [u^n]\prod_{1\leq k\leq R\sqrt n}
  \frac{1}{1-\exp(-k^2/(2n))u^k},
\]
up to the same factor \(\exp(O_R(1))\).  This factor is negligible on the
\(\sqrt n\) logarithmic scale.

The proof of Theorem \ref{thm:quadratic-model} applies with the Riemann sums
truncated to \(0\leq x\leq R\).  The truncated triangular array still has
variance \(\Theta(n^{3/2})\); after choosing the saddle \(c_R\), its variance
constant is
\[
  \int_0^R
  \frac{x^2\exp(c_Rx+x^2/2)}
       {(\exp(c_Rx+x^2/2)-1)^2}\,dx>0.
\]
Thus the same local estimate contributes only \(e^{o(\sqrt n)}\).  The
coefficient has
logarithm \(C_R\sqrt n+o(\sqrt n)\), where the saddle equation is the
displayed equation defining \(c_R\).  The distinct-part proof is identical
with the truncated Fermi product
\[
  [u^n]\prod_{1\leq k\leq R\sqrt n}
  (1+\exp(-k^2/(2n))u^k).
\]
The condition \(R>\sqrt2\) is exactly the feasibility condition
\(\int_0^R x\,dx>1\) for the limiting distinct-part saddle.

Finally, monotone convergence gives \(C_R\to C\) and \(C_{R,d}\to C_d\).
Since \(B_R(n)\leq A372395(n)\) and
\(B_R^{\rm dist}(n)\leq A370613(n)\), letting \(R\to\infty\) in the
truncated lower bounds gives the two displayed liminf estimates.
\end{proof}

\begin{theorem}[far-tail upper bound]\label{thm:far-tail}
For \(A>0\), put
\[
  U_A(n)=
  \sum_{\substack{\lambda\vdash n\\ \lambda_1\geq A n^{3/4}}}
  \AO(K_\lambda)
\]
and define \(U_A^{\rm dist}(n)\) analogously with the sum restricted to
distinct-part partitions.  Then
\[
  \limsup_{n\to\infty}
  \frac1{\sqrt n}\log\frac{U_A(n)}{n!}
  \leq
  \pi\sqrt{\frac23}-\frac{A^2}{2},
\]
and
\[
  \limsup_{n\to\infty}
  \frac1{\sqrt n}\log\frac{U_A^{\rm dist}(n)}{n!}
  \leq
  \frac{\pi}{\sqrt3}-\frac{A^2}{2}.
\]
\end{theorem}

\begin{proof}
Let \(L=\lambda_1\).  Splitting every part other than the largest into
singletons changes vertices that were previously in the same independent
part into vertices in different parts, and therefore adds precisely the
edges among those vertices while leaving all existing edges in place.  Thus
\(K_\lambda\) is a spanning subgraph of \(K_{L,1,\ldots,1}\).  Lemma
\ref{lem:edge-monotonicity} therefore gives
\[
  \AO(K_\lambda)\leq \AO(K_{L,1,\ldots,1}),
\]
where there are \(n-L\) singleton parts.

By Lemma \ref{lem:one-large-part},
\[
  \AO(K_{L,1,\ldots,1})=(n-L)!(n-L+1)^L.
\]
Consequently
\[
  \log\frac{\AO(K_\lambda)}{n!}
  \leq
  L\log(n-L+1)+\log((n-L)!)-\log(n!).
\]
Uniformly for \(L=o(n)\), the right side is
\[
  -\frac{L^2}{2n}+O\!\left(\frac{L^3}{n^2}+\frac{L}{n}\right),
\]
while for \(L\geq \varepsilon n\) it is \(-\Omega_\varepsilon(n)\).  Thus,
uniformly over \(L\geq A n^{3/4}\),
\[
  \frac{\AO(K_\lambda)}{n!}
  \leq
  \exp\!\left(-\left(\frac{A^2}{2}+o(1)\right)\sqrt n\right).
\]
The Hardy--Ramanujan estimate
\(\log p(n)\sim \pi\sqrt{2n/3}\) \cite{hardyramanujan1918,andrews1998},
multiplied with
the preceding bound, proves the first inequality.  The distinct-part
inequality follows in the same way from
\(\log q(n)\sim \pi\sqrt{n/3}\), where \(q(n)\) is the number of partitions
of \(n\) into distinct parts.
\end{proof}

\begin{lemma}[Laplace bound with quadratic penalty]\label{lem:penalized-gamma}
Fix \(0<\beta<1\).  Uniformly for \(0\leq Q\leq n^{1+\beta}\) and for
\(\gamma\) in a fixed compact subinterval of \((0,\infty)\),
\[
  \int_0^\infty e^{-t}t^n\exp\!\left(-\frac{\gamma Q}{t}\right)dt
  \leq
  n!\,
  \exp\!\left(
    -\frac{\gamma Q}{n}
    +O\!\left(n^{\beta-1}\frac{Q}{n}+\log n\right)
  \right).
\]
\end{lemma}

\begin{proof}
Let \(q=\gamma Q\) and
\[
  f_q(t)=-t+n\log t-\frac{q}{t}.
\]
The maximum is attained at
\[
  t_q=\frac{n+\sqrt{n^2+4q}}2
  =
  n+\frac{q}{n}+O(q^2/n^3),
\]
because \(q\leq O(n^{1+\beta})=o(n^2)\).  Substitution gives
\[
  f_q(t_q)
  =
  -n+n\log n-\frac{q}{n}
  +O(q^2/n^3).
\]
Since \(q^2/n^3=O(n^{\beta-1}q/n)\), Laplace's elementary upper bound
gives only a polynomial loss.  Indeed,
\[
  -f_q''(t_q)=\frac{n}{t_q^2}+\frac{2q}{t_q^3}\asymp \frac1n,
\]
because \(t_q=n+o(n)\).  On \([t_q/2,2t_q]\), Taylor's theorem and
concavity bound the integral by \(O(\sqrt n)\exp(f_q(t_q))\), while outside
this interval the drop from the maximum is \(\Omega(n)\).  Therefore
\(\int e^{f_q(t)}dt\leq \exp(f_q(t_q)+O(\log n))\), and Stirling's formula
for \(n!\) gives the stated estimate.  The polynomial factor is absorbed by
the \(O(\log n)\) term.
\end{proof}

\begin{lemma}[growing largest-part window]\label{lem:growing-window-paper1}
Let \(M_n=n^\beta\) with \(3/4<\beta<1\).  Uniformly over partitions
\(\lambda\vdash n\) with \(\lambda_1\leq M_n\),
\[
  \frac{\AO(K_\lambda)}{n!}
  \leq
  \exp(o(\sqrt n))
  \exp\!\left(-\eta_n\frac{Q(\lambda)}{2n}\right),
\]
where \(Q(\lambda)=\sum_i\lambda_i^2\) and
\[
  \eta_n=1-O(n^{\beta-1}+n^{-1/2}\log n).
\]
Here \(\beta\) is fixed before \(n\to\infty\), so \(\eta_n\to1\).  The same
estimate holds for distinct-part partitions.
\end{lemma}

\begin{proof}
Lemma \ref{lem:uniform-Pm} gives, uniformly when \(m\leq M_n=o(n)\) and
\(t\in[n/2,2n]\),
\[
  \log\frac{P_m(t)}{t^m}
  \leq
  -\frac{m(m-1)}{2t}
  +O\!\left(\frac{m^3}{t^2}+\frac{m}{t^2}\right).       \tag{7}
\]
It also gives \(P_m(t)>0\) on this central range for all sufficiently large
\(n\), so the logarithms are applied to positive quantities.  Summing (7)
over the parts of \(\lambda\) gives
\[
  \prod_iP_{\lambda_i}(t)
  \leq
  t^n
  \exp\!\left(
    -\frac{Q(\lambda)-n}{2t}
    +O\!\left(\frac{M_nQ(\lambda)}{t^2}+1\right)
  \right).
\]
Since \(Q(\lambda)\leq M_nn\), the error in the exponent is
\[
  O\!\left(\frac{M_nQ(\lambda)}{n^2}+1\right)
  =
  O(n^{\beta-1})\frac{Q(\lambda)}{n}+O(1)
\]
for \(t\in[n/2,2n]\).  The tails are controlled in absolute value as follows.
For \(t\leq n/2\), Lemma \ref{lem:uniform-Pm} gives
\[
  \prod_i |P_{\lambda_i}(t)|\leq (t+M_n)^n,
\]
and since \(M_n=o(n)\), the exponent \(-t+n\log(t+M_n)\) is smaller than the
central Gamma exponent \(-n+n\log n\) by \(\Omega(n)\).  For \(t\geq2n\),
the same crude bound gives
\[
  (t+M_n)^n\leq t^n\exp(O(nM_n/t))\leq t^n\exp(O(M_n)),
\]
and the upper Gamma tail is \(\exp(-\Omega(n))\), so this remains negligible
on the \(\sqrt n\) scale.  Hence the parts of the integral outside
\([n/2,2n]\) are negligible on the \(\sqrt n\) scale.

On the central range the preceding product estimate can be written as
\[
  \prod_iP_{\lambda_i}(t)
  \leq
  t^n
  \exp(O(1))
  \exp\!\left(
    -\frac{\gamma_n Q(\lambda)}{t}
  \right),
\]
where \(\gamma_n=\frac12(1-O(n^{\beta-1}))\).  Applying Lemma
\ref{lem:penalized-gamma} with \(Q=Q(\lambda)\) gives
\[
  \frac{\AO(K_\lambda)}{n!}
  \leq
  \exp(o(\sqrt n))
  \exp\!\left(
    -\left(1-O(n^{\beta-1})\right)\frac{Q(\lambda)}{2n}
  \right).
\]
This is the displayed estimate, after harmlessly enlarging the error in
\(\eta_n\).  No uniformity as \(\beta\uparrow1\) is asserted or needed.
\end{proof}

\begin{proof}[Proof of Theorem \ref{thm:partition-sums}]
The lower bounds are exactly the last assertion of
Theorem \ref{thm:truncated-partition}, obtained by first fixing \(R\) and
then letting \(R\to\infty\).

For the upper bound, fix once and for all a number \(3/4<\beta<1\), say
\(\beta=7/8\), and put \(M_n=n^\beta\).  Split the
unrestricted sum \(B(n)\) according to whether \(\lambda_1\leq M_n\).  On the
first range, Lemma \ref{lem:growing-window-paper1} gives
\[
  \frac{1}{n!}\sum_{\substack{\lambda\vdash n\\ \lambda_1\leq M_n}}
  \AO(K_\lambda)
  \leq
  \exp(o(\sqrt n))
  [u^n]\prod_{k\geq1}
  \frac{1}{1-\exp(-\eta_n k^2/(2n))u^k},
\]
where \(\eta_n=1-O(n^{\beta-1}+n^{-1/2}\log n)\to1\).
The proof of Theorem \ref{thm:quadratic-model}, uniformly for the quadratic
coefficient in a compact neighborhood of \(1\), shows that the logarithm of
the coefficient is at most \((C+o(1))\sqrt n\).  Explicitly, replace
\(x^2/2\) in the quadratic model by \(\eta x^2/2\), with
\(\eta\in[1-\delta,1+\delta]\), and set
\[
  I(a,\eta)=
  \int_0^\infty\frac{x}{\exp(a x+\eta x^2/2)-1}\,dx .
\]
For \(\eta\) near \(1\), the saddle \(a=c(\eta)\) is defined by
\(I(c(\eta),\eta)=1\).  Lemma \ref{lem:partition-constants} gives
\(\partial I/\partial a<0\) at \((c,1)\).  The integral is continuous in
\(\eta\) by dominated convergence: near the origin the integrand is bounded
uniformly on a small compact neighborhood of \((c,1)\), and at infinity it is
bounded by an integrable exponential tail.  Hence the implicit-function
theorem gives continuity of \(c(\eta)\) at \(\eta=1\).  The same domination
applied to
\[
  C(\eta)=c(\eta)+\int_0^\infty
  -\log(1-\exp(-c(\eta)x-\eta x^2/2))\,dx .
\]
shows that \(C(\eta)\) is continuous at \(\eta=1\).  Thus a \(1+o(1)\)
perturbation of the quadratic coefficient changes the leading logarithmic
constant by \(o(1)\).  No rate of convergence for \(\eta_n\to1\) is needed:
for every fixed \(\varepsilon>0\), continuity gives
\(C(\eta_n)\leq C+\varepsilon\) for all sufficiently large \(n\), and then
\(\varepsilon\downarrow0\).

On the complementary range \(\lambda_1>M_n\), fix an arbitrary constant
\(A>0\).  Since \(M_n=n^\beta\) and \(\beta>3/4\), for all sufficiently large
\(n\) the inclusion
\[
  \{\lambda_1>M_n\}\subseteq \{\lambda_1>A n^{3/4}\}
\]
holds.  Theorem \ref{thm:far-tail} therefore gives
\[
  \limsup_{n\to\infty}\frac1{\sqrt n}
  \log\frac{1}{n!}
  \sum_{\substack{\lambda\vdash n\\\lambda_1>M_n}}\AO(K_\lambda)
  \leq
  \pi\sqrt{\frac23}-\frac{A^2}{2}.
\]
Letting \(A\to\infty\) shows that this range is exponentially smaller on
the \(\sqrt n\) scale.  Thus
\[
  \limsup_{n\to\infty}
  \frac1{\sqrt n}\log\frac{B(n)}{n!}\leq C.
\]
Together with the lower bound this proves the unrestricted formula.

For the distinct-part sum, the same argument uses the Fermi product
\[
  \prod_{k\geq1}(1+\exp(-\eta_n k^2/(2n))u^k)
\]
and the distinct-part estimate in Theorem \ref{thm:far-tail}.  This gives
the upper bound \(C_d\), while Theorem \ref{thm:truncated-partition} gives
the matching lower bound.
\end{proof}

\section{Numerical checks}

The following checks use exact arithmetic for the acyclic-orientation counts
and then compare them with the asymptotic formulae above.
The constants \(c,C,c_d,C_d\) were computed by high-precision numerical
quadrature of the defining integrals and Newton iteration for the saddle
equations.  With 80 decimal digits of working precision this gives residuals
below \(10^{-70}\) for both equations \(I(c)=1\) and \(I_d(c_d)=1\).  The
tables below provide numerical consistency checks for the logarithmic
asymptotics; the partition-sum convergence is slow.  The computations are
reproducible from the source-bundle scripts
\begin{center}
\begin{tabular}{l}
\path{computations/verify_a372326_columns.py}\\
\path{computations/verify_a372084_diagonal.py}\\
\path{computations/verify_fixed_proportions.py}\\
\path{computations/verify_rectangular_window.py}\\
\path{computations/equal_size_expansion_coefficients.py}\\
\path{computations/verify_equal_size_critical_windows.py}\\
\path{computations/explore_blowup_base_graphs.py}\\
\path{computations/verify_quadratic_partition_model.py}\\
\path{computations/explore_a372395_partition_sums.py}
\end{tabular}
\end{center}
The listed scripts are included in the prepared manuscript source bundle and
can be supplied with the source files or on request.
The coefficient script symbolically reproduces the rational constants
\(\gamma_j\) in Theorem \ref{thm:all-poly-equal} from the near-diagonal
Stirling numbers \(S(m,m-k)\).
The blow-up exploration script reports numerical Hessian spectra and coarse
torus searches for small fixed base graphs; it is exploratory only, not an
input to any proof.
The last script expands the polynomials \(P_m(t)\) exactly and evaluates the
Gamma integrals term by term, using \(\int_0^\infty e^{-t}t^jdt=j!\).

\[
\begin{array}{c|c|c|c}
k&r& C_k(r)/(e^{-(k-1)/2}(kr)!)&
\displaystyle
\frac{C_k(r)}
{e^{-(k-1)/2}(kr)!\left(1-\frac{(k-1)(5k-7)}{24kr}\right)}\\
\hline
2&100&0.999372058537&0.999997056698\\
3&100&0.997769105145&0.999991308051\\
4&100&0.995924621674&0.999987069142\\
5&100&0.993985977785&0.999985893144\\
2&200&0.999686766107&0.999999265878\\
3&200&0.998886724564&0.999997833268\\
4&200&0.997965534982&0.999996778439\\
5&200&0.996996497211&0.999996486671
\end{array}
\]

For the quadratic-energy model in Theorem \ref{thm:quadratic-model}, exact
dynamic programming for the products gives:
\[
\begin{array}{c|c|c|c|c}
n&\log Z_n/\sqrt n& C&\log Z_n^{\rm dist}/\sqrt n&C_d\\
\hline
100&1.481354907392&2.158752005658&0.424552109924&0.905729821720\\
200&1.631617363138&2.158752005658&0.529028810733&0.905729821720\\
500&1.784872744083&2.158752005658&0.636916929791&0.905729821720\\
1000&1.872633831819&2.158752005658&0.699269831515&0.905729821720\\
2000&1.941023424079&2.158752005658&0.748145364362&0.905729821720\\
5000&2.008138720630&2.158752005658&0.796362248874&0.905729821720\\
10000&2.045338197965&2.158752005658&0.823202459873&0.905729821720
\end{array}
\]
The convergence is visibly slow: the theorem proves only an \(o(\sqrt n)\)
error in the logarithm, and the Riemann-sum corrections near the logarithmic
singularity at the origin remain substantial at these sizes.

The truncated constants in Theorem \ref{thm:truncated-partition} converge
quickly to the full constants of Theorem \ref{thm:partition-sums}:
\[
\begin{array}{c|c|c|c|c}
R&c_R&C_R&c_{R,d}&C_{R,d}\\
\hline
2&0.748267431442&2.148930549281&-0.749557358925&0.732080500591\\
3&0.764331079876&2.158471927044&-0.353316990400&0.895495758756\\
4&0.764986685523&2.158748834028&-0.324835266833&0.905416631365\\
5&0.764996390340&2.158751991944&-0.323714467806&0.905725963815
\end{array}
\]

Direct exact computations for the actual OEIS partition sums are more
expensive, but they show the same upward drift toward the constants:
\[
\begin{array}{c|c|c|c|c}
n&\log(A372395(n)/n!)/\sqrt n&C&
\log(A370613(n)/n!)/\sqrt n&C_d\\
\hline
40&1.282653368698&2.158752005658&0.233195614083&0.905729821720\\
60&1.392234736244&2.158752005658&0.316501953286&0.905729821720\\
80&1.463369443608&2.158752005658&0.370820908037&0.905729821720\\
100&1.514838417821&2.158752005658&0.410298875884&0.905729821720\\
120&1.554547425786&2.158752005658&0.440805617478&0.905729821720
\end{array}
\]
This slow convergence is consistent with the \(o(\sqrt n)\) error term in
Theorem \ref{thm:partition-sums} and with the size of the local-limit
corrections near the logarithmic singularity at the origin.

For the rectangular window, the table gives
\[
  \log(D_{m,n}/(mn)!)+m/2
\]
and the limiting target \(1/2-5m/(24n)\), with \(m=\kappa n\).
\[
\begin{array}{c|c|c|c}
\kappa&n&\log(D_{\kappa n,n}/(\kappa n^2)!)+\kappa n/2&
1/2-5\kappa/24\\
\hline
1&20&0.310395599957&0.291666666667\\
1&30&0.304159077535&0.291666666667\\
2&20&0.095975023893&0.083333333333\\
2&30&0.091728947875&0.083333333333\\
3&20&-0.118704935890&-0.125000000000\\
3&30&-0.120813974001&-0.125000000000
\end{array}
\]
Additional sublinear and intermediate checks report the absolute logarithmic
error
\[
  \log D_{m,n}-\log\left((mn)!
  \exp(-m/2+1/2-5m/(24n))\right).
\]
\[
\begin{array}{c|c|c}
m&n&\hbox{logarithmic error}\\
\hline
4&80&0.005305364556\\
8&80&0.005688794806\\
12&80&0.005764203787\\
10&100&0.004616635350\\
20&100&0.004638397209\\
40&100&0.004460485030
\end{array}
\]

The next table checks the enlarged equal-size expansion at the new critical
scales.  It reports
\[
  R_{m,n}=
  \log(D_{m,n}/(mn)!)+\frac m2-\frac12+\frac{5m}{24n}
  +\frac{m}{8n^2}+\frac{251m}{2880n^3}.
\]
The theorem predicts \(R_{m,n}=O(n^{-1})\), and the final column displays
the stabilized product \(nR_{m,n}\).
\[
\begin{array}{c|c|c|c|c}
\hbox{scale}&n&m&R_{m,n}&nR_{m,n}\\
\hline
n^2&4&16&0.144129537270&0.576518149080\\
n^2&5&25&0.112323371932&0.561616859661\\
n^2&6&36&0.091913153040&0.551478918240\\
n^2&7&49&0.077739289446&0.544175026124\\
n^3&3&27&0.192787311739&0.578361935217\\
n^3&4&64&0.133763861927&0.535055447708\\
n^3&5&125&0.102289218067&0.511446090334
\end{array}
\]

For Theorem \ref{thm:tutte-axis}, exact chromatic-polynomial evaluation gives
the following ratios \(H_s(N,p)\) divided by the asymptotic formula.
\[
\begin{array}{c|c|c|c}
p&s&N&\hbox{ratio}\\
\hline
2&2&40&0.9585961408\\
2&2&120&0.9851492927\\
3&2&40&0.9701833293\\
3&2&120&0.9948827955\\
4&3&40&1.0142368220\\
4&3&120&1.0042705860\\
5&2&40&0.9958415693\\
5&2&120&0.9984461997
\end{array}
\]

For Theorem \ref{thm:fixed-proportions}, exact bipartite values give:
\[
\begin{array}{c|c|c}
a&b&\AO(K_{a,b})/\hbox{asymptotic}\\
\hline
30&70&0.963630172154\\
40&80&0.976688955430\\
60&60&0.989496165924\\
80&120&0.991145418809\\
75&175&0.985096250805
\end{array}
\]

For the first correction in Theorem \ref{thm:tutte-product-regimes}, the
table compares the leading and corrected fixed part-size formulae:
\[
\begin{array}{c|c|c|c|c}
k&s&r&\hbox{leading ratio}&\hbox{corrected ratio}\\
\hline
2&2&40&1.0171788543&0.9999915004\\
2&2&120&1.0057282135&0.9999990523\\
3&2&40&1.0110565510&0.9999460394\\
3&2&120&1.0036976822&0.9999940007\\
4&3&40&1.0274180099&1.0000722834\\
4&3&120&1.0091229629&1.0000083039
\end{array}
\]

\section{Open problems}

The partition-sum problem for A372395 and A370613 is settled here at the
logarithmic scale.  Several natural refinements remain.

\begin{enumerate}
\item Determine the full prefactor, and preferably a complete asymptotic
expansion, for A372395 and A370613.  The proof above identifies the
logarithmic constant but deliberately avoids the local two-variable saddle
analysis needed for a polynomial prefactor.
\item Extend the Tutte-axis results away from \(T_G(1+s,0)\).  In particular,
it would be valuable to understand whether analogous integral or
coefficient representations can treat \(T_G(x,y)\) with \(y\neq0\) for
complete multipartite graphs.  The present method factors through the
chromatic specialization \(T_G(1-q,0)\); for \(y\neq0\), the Tutte polynomial
remembers connected spanning subgraphs and no longer decomposes only by
proper color classes.  A plausible route would need an additional component
counter in the integral or a random-cluster representation before the
saddle-point analysis can begin.
\item Develop a \(q\)-analogue of the orientation counts, for example through
sink enumerators or chromatic symmetric functions, and determine whether the
same saddle-point regimes persist.  The chromatic symmetric function
framework of Stanley \cite{stanley1995} is the natural starting point.
\item Develop a complete two-parameter phase diagram for
\(\AO(K_{m,\ldots,m})\) when both \(m,n\to\infty\), beyond every fixed
polynomial window \(m\leq Cn^J\), and obtain higher-order corrections in the
fixed-proportion theorem on the Tutte axis.
\item Extend the fixed graph blow-up framework beyond conditional
smooth-point ACSV, for example by proving strict minimality for broad
families of base graphs.  Lemma
\ref{lem:complete-base-critical-param} shows that even the complete-base
fixed-proportion problem involves branch choices for the unimodal function
\(-z\log z\), while the \(C_5\) specialization in Section 3 gives a natural
first non-complete test case.  Another direction is to treat graph classes
such as threshold graphs and cographs whose chromatic polynomials have
recursive decompositions rather than a single fixed-base blow-up form.
\end{enumerate}

\section{Related work and priority notes}

The identity \(\AO(G)=(-1)^{|V(G)|}\chi_G(-1)\) is due to Stanley
\cite{stanley1973}.  Exact enumerative formulae for acyclic orientations of
complete multipartite graphs are discussed in recent work of Carballosa,
Reyes, and Khera \cite{carballosa2025}.  Their work gives exact encodings
and closed formulae rather than the asymptotic regimes studied here.

The complete bipartite case has an additional connection with
poly-Bernoulli numbers and lonesum matrices.  Cameron, Glass, Rekv{\'e}nyi,
and Schumacher showed that \(\AO(K_{a,b})\) is the negative-index
poly-Bernoulli number \(B_a^{(-b)}\), equivalently the number of \(a\times b\)
lonesum matrices \cite{cameronglassrekvenyi2014}.  Khera, Lundberg, and
Melczer subsequently used ACSV to obtain bivariate asymptotics for these
poly-Bernoulli numbers and lonesum matrices \cite{kheralundbergmelczer2021}.
Thus the \(p=2\) acyclic-orientation specialization of the fixed-column and
fixed-proportion ACSV analysis here is equivalent to, or follows from, this
earlier bipartite/lonesum-matrix theory.  The new fixed-column contribution of the
present paper is the complete multipartite extension to arbitrary fixed
\(p\geq3\), together with the Tutte-axis variants, product regimes,
partition-sum asymptotics, and fixed graph blow-up framework developed here.
More recently, Khera and Lundberg studied the distribution of the length of
the longest path in random acyclic orientations of \(K_{a,b}\), using related
generating functions and analytic-combinatorial methods
\cite{kheralundberg2025}.

A recent unpublished conference abstract of Lundberg, Rodriguez, and Corcoran
\cite{corcoran2026} announces ACSV asymptotics for the equal-size complete
tripartite case.  That case is contained in the specialization \(p=3\) of
Theorem \ref{thm:a267383}; the present manuscript records the overlap
explicitly and treats arbitrary fixed \(p\), conditional fixed-proportion
saddles, product windows, Tutte-axis evaluations, and partition sums.  OEIS
A267383 records Kotesovec's fixed-column asymptotic as a conjecture.  Apart
from the bipartite case discussed above and the announced tripartite
equal-size case, the cited exact enumeration literature does not appear to
contain the arbitrary fixed-\(p\geq3\) fixed-column formula, the Tutte-axis
variants, the product-window asymptotics, or the partition-sum asymptotics
proved here.

Classical and modern tools for weighted partition asymptotics include
Wright's circle method \cite{wright1934} and Meinardus-type theorems,
including the probabilistic extensions of Granovsky, Stark, and Erlihson
\cite{granovsky2008} and later developments \cite{granovsky2015}.  Our
partition-sum proof uses this standard Khintchine and local-limit framework, but
the weights arise from the graph-polynomial integral and therefore require
the separate growing-window comparison and far-tail estimate before the
quadratic-energy model can be transferred back to A372395 and A370613.

\appendix

\section{OEIS-ready consequences}

For OEIS use, the main consequences can be summarized as follows; the proofs
are the corresponding theorems in the body of the paper.
\begin{description}
\item[A267383.]
For fixed \(p\ge2\), with \(L=\log(p/(p-1))\),
\[
  \AO(T(N,p))\sim
  \frac{N!}{(p-1)(1-L)^{(p-1)/2}p^NL^{N+1}}.
\]
\item[A372326.]
For fixed \(k\),
\[
  \AO(K_{\underbrace{k,\ldots,k}_{r}})
  =
  e^{-(k-1)/2}(kr)!
  \left(1-\frac{(k-1)(5k-7)}{24kr}+O(r^{-2})\right).
\]
For equal-size parts, uniformly for \(1\le m\le Cn^3\),
\[
  \AO(K_{\underbrace{m,\ldots,m}_{n}})
  =(mn)!\exp\!\left(
  -\frac m2+\frac12-\frac{5m}{24n}
  -\frac{m}{8n^2}-\frac{251m}{2880n^3}+O_C(n^{-1})\right).
\]
\item[A372084.]
The main diagonal \(A372326(n,n)=A267383(n^2,n)\) satisfies
\[
  \AO(K_{\underbrace{n,\ldots,n}_{n}})
  =
  (n^2)!\exp\!\left(-\frac n2+\frac7{24}+O(n^{-1})\right).
\]
\end{description}

\section{Notation index}

\[
\begin{array}{c|l}
\hbox{symbol} & \hbox{meaning}\\
\hline
\AO(G) & \hbox{number of acyclic orientations of }G\\
\chi_G(q) & \hbox{chromatic polynomial of }G\\
T_G(x,y) & \hbox{Tutte polynomial of }G\\
K_{\lambda_1,\ldots,\lambda_r} & \hbox{complete multipartite graph with
part sizes }\lambda_i\\
H[\lambda] & \hbox{independent-set blow-up of a fixed base graph }H\\
B_H(x) & \sum_{I\in\mathrm{Ind}(H)}\prod_{i\in I}(e^{x_i}-1)\\
F_H(y) & B_H(-y)\\
N & \lambda_1+\cdots+\lambda_r,\ \hbox{the total number of vertices}\\
S(m,j) & \hbox{Stirling number of the second kind}\\
P_m(t) & \sum_{j=1}^m(-1)^{m+j}S(m,j)t^j\\
(q)_J & q(q-1)\cdots(q-J+1),\ \hbox{falling factorial}\\
H_s(G) & (-1)^{|V(G)|}\chi_G(-s)=sT_G(1+s,0)\ \hbox{for connected }G\\
L & \log(p/(p-1))\\
Q(\lambda) & \sum_i\lambda_i^2\\
B(n),B_d(n) & \hbox{A372395 and A370613 partition sums}\\
Z_n,Z_n^{\rm dist} & \hbox{quadratic-energy unrestricted and distinct models}\\
c,C,c_d,C_d & \hbox{partition-saddle constants in Theorem
\ref{thm:partition-sums}}\\
M_n & \hbox{growing largest-part cutoff in Lemma
\ref{lem:growing-window-paper1}}
\end{array}
\]

\section*{Funding}

This research did not receive any specific grant from funding agencies in the
public, commercial, or not-for-profit sectors.

\section*{Declaration of generative AI and AI-assisted technologies in the
manuscript preparation process}

During the preparation of this work the author used ChatGPT as an assistive
tool for language editing, organization of exploratory algebraic notes, and
code-drafting support.  After using this tool, the author reviewed and edited
the content as needed and takes full responsibility for the content of the
published article.  All mathematical arguments, references, and numerical
checks reported in Section 10 were independently verified by the author; the
numerical checks are based on exact arithmetic or dynamic-programming
evaluations of the formulae stated in the paper.  No generative AI or
AI-assisted tools were used to create or alter figures, images, or artwork.

\end{document}